\documentclass[A4paper,10pt]{amsart}
\usepackage{amsmath,amssymb,latexsym,amsfonts}
\usepackage[catalan,english]{babel}
\usepackage[T1]{fontenc}
\usepackage{upquote}
 \theoremstyle{plain}    
 \newtheorem{theorem}{Theorem}[section]
 \numberwithin{equation}{section} 
 \numberwithin{figure}{section} 
 \theoremstyle{plain}
 \theoremstyle{plain}    
 \newtheorem{corollary}[theorem]{Corollary} 
 \theoremstyle{plain}    
 \newtheorem{prop}[theorem]{Proposition} 
 \theoremstyle{plain}    
 \newtheorem{lemma}[theorem]{Lemma} 
 \theoremstyle{remark}
 \newtheorem{rmk}[theorem]{Remark}
 \theoremstyle{definition}
\newtheorem{examp}[theorem]{Example}
\newtheorem{definition}[theorem]{Definition}
\theoremstyle{plain}

\newcommand{\E}{{\rm E}}
\newcommand{\N}{\mathbb{N}}
\newcommand{\R}{\mathbb{R}}
\newcommand{\Rinf}{\mathbb{R}\cup \{-\infty\}}
\newcommand{\C}{\mathbb{C}}
\newcommand{\Cc}{\mathcal{C}}
\newcommand{\f}{\varphi}
\newcommand{\p}{\psi}
\newcommand{\Ca}{{\rm Cap}}
\newcommand{\Ec}{\mathcal{E}}
\newcommand{\Fc}{\mathcal{F}}
\newcommand{\Eco}{\Ec_m(\Omega)}
\newcommand{\Fco}{\Fc_m(\Omega)}
\newcommand{\Ecto}{\Ec^1_m(\Omega)}

\newcommand{\Ecpo}{\Ec_m^p(\Omega)}
\newcommand{\Fcpo}{\Fc_m^p(\Omega)}
\newcommand{\Fcao}{\Fc_m^a(\Omega)}
\newcommand{\Ecoo}{\Ec_m^0(\Omega)}
\newcommand{\Mcc}{\mathcal{M}}
\newcommand{\SH}{\mathcal{SH}}
\newcommand{\um}{u_{m,E,\Omega}}
\newcommand{\ind}{1{\hskip -2.5 pt}\hbox{\textsc{I}}}
\newcommand{\pO}{\partial \Omega}
\newcommand{\weak}{\rightharpoonup}
\newcommand{\loc}{L^{\infty}_{\rm loc}}
\newcommand{\into}{\int_{\Omega}}

\newcommand{\Fcmu}{\mathcal{F}_{\mu}}
\newcommand{\Lmu}{_{L^1(\mu)}}
\newcommand{\Lcm}{\mathcal{L}_{\mu}}
\usepackage{hyperref}
\usepackage{xcolor}
\usepackage{comment}
\usepackage{dsfont}
\usepackage{yhmath}
\definecolor{violet}{rgb}{0.0,0.2,0.7}
\definecolor{rouge}{cmyk}{0.0,0.6,0.4,0.3}
\definecolor{rouge2}{rgb}{0.8,0.0,0.2}
\hypersetup{
    bookmarks=true,         
    unicode=false,          
    pdftoolbar=true,        
    pdfmenubar=true,        
    pdffitwindow=false,     
    pdfstartview={FitH},    
    pdftitle={},    
    pdfauthor={},     
    colorlinks=true,       
   linkcolor=black,          
    citecolor=black,        
    filecolor=black,      
    urlcolor=cyan}           

\title[Variational approach to  Hessian equations]
 {A variational approach to complex Hessian equations in $\mathbb{C}^n$ } 

\author{Lu Hoang Chinh}

\begin{document}
\maketitle

\begin{abstract}
Let $\Omega$ be a $m$-hyperconvex domain of $\C^n$  and $\beta$ be the standard K\"{a}hler form in $\C^n$. We introduce finite energy classes of $m$-subharmonic functions of Cegrell type, $\Ecpo, p>0$ and $\Fco$. Using a variational method  we show that the degenerate complex Hessian equation  $(dd^c\varphi)^m\wedge \beta^{n-m}=\mu$ has a unique solution in $\Ecto$ if and only if every function in $\Ecto$ is integrable with respect to $\mu$. If $\mu$ has finite total mass and does not charge $m$-polar sets, then the equation  has a unique solution in $\Fco$. 
\end{abstract}

\section{Introduction}
Let $\Omega$ be a bounded domain of $\C^n$ and $m$ be an integer such that $1\leq m\leq n.$ We consider complex $m$-Hessian equations of the form 
\begin{equation}\label{eq: heq}
(dd^c \f)^m\wedge \beta^{n-m}=\mu,
\end{equation}
where $\beta:=dd^c \vert z\vert^2$ is the standard K\"{a}hler form in $\C^n$ and $\mu$ is a positive Radon measure. 

The border cases $m=1$ and $m=n$ correspond to the Laplace equation which is a classical subject and the complex Monge-Amp\`ere equation which was studied intensively in the recent years by many authors. 

\medskip

The complex $m$-Hessian equation was first studied by Li \cite{Li04}. He used the well-known continuity method to solve the non-degenerate Dirichlet problem for equation (\ref{eq: heq}) (where the data is smooth and we seek for smooth solutions) in strongly $m$-pseudoconvex domains. One of its  degenerate  counterparts was studied by B{\l}ocki \cite{Bl05}. More precisely, he solved the homogeneous equation with continuous boundary data and initiated a potential theory for this equation. Recently, Abdullaev and Sadullaev \cite{SA12} also considered  $m$-polar sets and $m$-capacity for $m$-subharmonic functions. When the right-hand side $\mu$ has density in $L^p(\Omega)$ ($p>n/m$) Dinew and Ko{\l}odziej proved in  \cite{DK11} that given a continuous boundary data, the Dirichlet problem of equation (\ref{eq: heq}) has a unique continuous solution. The H\"{o}lder regularity of the solution has been recently studied by Nguyen Ngoc Cuong \cite{Cuong2}. He also showed how to construct solutions from subsolutions \cite{Cuong1}. A viscosity approach to this equation has been developed in \cite{Chinh2} which generalize results in \cite{YW10} and \cite{EGZ11}.
\medskip

The \textit{real} Hessian equation is a classical subject which was studied  intensively in the recent years. The reader can find a survey for this in \cite{W09}. It was explained in \cite{DK11} that \textit{real} and \textit{complex} Hessian equations are very different and direct adaptations of the real methods to the complex setting often fails.
\medskip

The corresponding complex $m$-Hessian equation on compact K\"{a}hler manifolds has been studied by many authors. It has the following form
\begin{equation}\label{eq: heq kahler}
 (\omega+dd^c \f)^m\wedge \omega^{n-m}=\mu,
\end{equation}
where $(X,\omega)$ is a compact K\"{a}hler manifold of dimension $n$ and $1\leq m\leq n$ and $\mu$ is a positive Radon measure.

When $\mu=f\omega^n$, $f>0$ is a smooth function satisfying the compatibility condition  $\int_X f\omega^n=\int_X \omega^n$, this is a generalization of the well-known Calabi-Yau equation \cite{Y}.  In \cite{Kok10}, Kokarev gave some conditions on the measure $\mu$ and on the holomorphic sectional curvature of the metric so that equation (\ref{eq: heq kahler}) has a $\omega$-plurisubharmonic solution. In general, if $\varphi$ solves  equation (\ref{eq: heq kahler}) the  form $\omega+dd^c \varphi$ is not positive. This lack of positivity prevents one from copying the proof of Yau's Theorem without assuming a positivity condition on the holomorphic bisectional curvature.  Hou, Ma, Wu \cite{Hou09}, and Jbilou \cite{Jb10}  independently proved that equation (\ref{eq: heq kahler}) has a  smooth solution  provided this positivity condition. Another effort from Hou, Ma and Wu \cite{HMW10} showed that one can obtain a $\Cc^2$ estimate if a gradient estimate holds. As suggested by these authors,  this estimate can be used in some blow-up analysis. This blow-up analysis reduces the problem of solving equation (\ref{eq: heq kahler}) to a Liouville-type theorem for $m$-subharmonic functions in $\C^n$ which was recently proved by Dinew and Ko{\l}odziej \cite{DK12} and the solvability of equation (\ref{eq: heq kahler}) is thus confirmed on any compact K\"{a}hler manifold.
\medskip

When $0\leq f\in L^p(X,\omega^n)$ for some $p>n/m,$ Dinew and Ko{\l}odziej recently proved that  (\ref{eq: heq kahler}) admits a unique continuous weak solution. The result also holds when the right-hand side $f=f(x,\f)$ depends on $\varphi$ (see \cite{Chinh1}). 
\medskip

To deal with more singular measures (measures of finite energy), the variational method developed in \cite{BBGZ13} is a powerful method. However,  applying this method to the complex Hessian equation (\ref{eq: heq kahler}) need further studies on the local Dirichlet problem and on the regularizing process which are not yet available and seem to be very difficult. 

As a matter of fact, it is interesting to first develop this approach  for the complex Hessian equation in the flat case, i.e the case when the metric is $\beta$. This is the aim of this paper. 

The paper is organized as follows. In section 2, we recall basic facts about $m$-subharmonic functions and the complex $m$-Hessian operators. At the end of section 2 we give a connection between the m-polarity and the Hausdorff measure of a set. Using this one can find examples of $m$-polar sets ($m<n$) which are not pluripolar.  In section 3, we study finite energy classes of $m$-subharmonic functions inspired by \cite{Ceg98,Ceg04}.  An $m$-subharmonic function $\varphi$ belongs to the class $\Ecto$  if  the Hessian measure $H_m(\varphi)=(dd^c \varphi)^m\wedge \beta^{n-m}$ is well-defined and with respect to which $\varphi$ is integrable. The class $\Fc_m^a(\Omega)$ consists of non-positive $m$-subharmonic functions whose Hessian measures are well-defined, of finite total mass and do not charge $m$-polar sets.
In section 4, we develop a variational approach inspired by \cite{BBGZ13} (see also \cite{ACC10}) to solve equation (\ref{eq: heq}) with a "finite energy" right-hand side.

\medskip

The main results are the followings.
\medskip

\noindent{\bf Theorem 1.}  Let $\mu$ be a positive Radon measure in $\Omega$, an $m$-hyperconvex domain. Then $\Ecto\subset L^1(\Omega,\mu)$ if and only if there exists a unique $\varphi\in \Ecto$ such that $(dd^c \varphi)^m\wedge \beta^{n-m}=\mu.$
\medskip

To prove this result we use a variational method introduced in \cite{BBGZ13}. Our result generalizes the result in \cite{ACC10}. Using this and following \cite{Ceg04} we also get:

\medskip

\noindent{\bf Theorem 2.}  Let $\mu$ be a positive Radon measure in an $m$-hyperconvex domain $\Omega$ such that $\mu(\Omega)<+\infty$ and $\mu$ does not charge $m$-polar sets. Then there exists a unique $\varphi\in \Fc^a_m(\Omega)$ such that $(dd^c \varphi)^m\wedge \beta^{n-m}=\mu.$

\section{Preliminaries} 

\subsection{m-subharmonic functions and the Hessian operator}
In the whole paper, $\beta$ denotes the standard K\"{a}hler form in $\C^n.$ In this section we summarize basic facts about $m$-subharmonic functions and the Hessian operator which will be used in the next sections. Most of these results can be found in \cite{Cuong1,Cuong2}, \cite{SA12} or can be proved  similarly as in the case of plurisubharmonic functions (see for example \cite{Kl91}, \cite{Kol05}).

\begin{definition}
Let $\alpha$ be a real $(1,1)$-form in $\Omega$, a domain of $\C^n$. We say that $\alpha$ is $m$-positive in $\Omega$ if the following inequalities hold 
$$
\alpha^j\wedge\beta^{n-j}\geq 0, \ \ \forall j=1,...,m.
$$
Let $T$ be a current of bidegree $(n-k,n-k)$, with $k\leq m$. Then $T$ is called $m$-positive if for all $m$-positive $(1,1)$-forms $\alpha_1,...,\alpha_k,$ we have
$$
\alpha_1\wedge...\wedge \alpha_k\wedge T\geq 0.
$$
\end{definition}

 \begin{definition}
A function $u: \Omega\rightarrow \R\cup \{-\infty\}$ is called $m$-subharmonic if it is subharmonic and the current $dd^c u$
is $m$-positive.
The class of all $m$-subharmonic functions in $\Omega$ will be denoted by $\SH_m(\Omega).$
 \end{definition}
 
 \begin{definition}
Let $\Omega$ be a  bounded  domain in $\C^n$. Then $\Omega$ is called $m$-hyperconvex if there exists a continuous $m$-subharmonic function $\f:\Omega\rightarrow \R^-$ such that $\{\f<c\}\Subset \Omega$, for every $c<0.$ 
\end{definition}

From now on, we always assume that $\Omega$ is $m$-hyperconvex.\\
We list in the following proposition some elementary facts on $m$-subharmonicity.
\begin{prop}\label{prop: basic property of m subharmonic functions}
(i) If $u$ is $\mathcal{C}^2$ smooth then $u$ is $m$-subharmonic if and only if the form $dd^cu$ is $m$-positive in $\Omega.$

(ii) If $u,v \in \SH_m(\Omega)$ then $\lambda u+\mu v\in \SH_m(\Omega), \forall \lambda, \mu >0.$

(iii) If $u$ is $m$-subharmonic in $\Omega$ then the standard regularization $u * \chi_{\epsilon}$ are also $m$-subharmonic in $\Omega_{\epsilon}:=\{x\in \Omega \ / \ d(x,\partial \Omega)>\epsilon\}$.

(iv) If $(u_l)\subset \SH_m(\Omega)$ is  locally uniformly bounded from above then $(\sup u_l)^*\in \SH_m(\Omega)$, where $v^*$ is the upper semi continuous regularization of $v$. 

(v) $PSH(\Omega)=\SH_n(\Omega)\subset ...\subset \SH_m(\Omega)\subset...\subset \SH_1(\Omega)=\SH(\Omega).$

(vi)  Let $\emptyset \neq U\subset \Omega$ be a proper open subset such that $\partial U\cap \Omega$ is relatively compact in $\Omega$. If $u\in \SH_m(\Omega)$, $v\in \SH_m(U)$ and $\limsup_{x\to y}v(x)\leq u(y)$ for each $y\in \partial U\cap \Omega$ then the function $w$, defined by
$$
w(z)= \begin{cases} u(z)\ \ \ \ \ \ \ \ \ \ \ \ \ \ \ \  {\rm if }\ z\in \Omega \setminus U ,\\
\max (u(z), v(z)) \ \ {\rm if} \ z\in U 
\end{cases}
$$
is $m$-subharmonic in $\Omega.$
\end{prop}

\noindent For locally bounded $m$-subharmonic functions $u_1,...,u_p$ (with $p\leq m$) we can inductively define a closed $m$-positive current (following Bedford and Taylor \cite{BT76}).
\begin{lemma}\label{lem: symmetric}
Let $u_1,...,u_k$ (with $k \leq m$) be locally bounded $m$-subharmonic functions in $\Omega$ and let $T$ be a closed $m$-positive current of bidegree $(n-p,n-p)$ (with $p\geq k$). Then we can define inductively a closed $m$-positive current 
$$
dd^cu_1\wedge dd^cu_2\wedge ...\wedge dd^cu_k\wedge T,
$$
and the product is symmetric, i.e. 
$$
dd^cu_1\wedge dd^cu_2\wedge ...\wedge dd^cu_p\wedge T=
dd^cu_{\sigma(1)}\wedge dd^cu_{\sigma(2)}\wedge ...\wedge dd^cu_{\sigma(k)}\wedge T,
$$
for every permutation $\sigma: \{1,...,k\}\to \{1,...,k\}.$

In particular, the Hessian measure of $u\in \SH_m(\Omega)\cap\loc$ is defined to be
$$
H_m(u)=(dd^cu)^m\wedge\beta^{n-m}.
$$
\end{lemma}
\begin{proof}
See \cite{SA12}.
\end{proof}

\begin{theorem}\label{thm: convergence decreasing}
Let $(u_0^j),...(u_k^j)$ be decreasing sequences of  $m$-subharmonic functions in $\Omega$ converging to $u_0,...,u_k\in \SH_m(\Omega)\cap \loc(\Omega)$ respectively. Let $T$ be a closed $m$-positive current of bidegree $(n-p,n-p)$ (with $p\geq k$) on $\Omega$. Then 
$$
u_0^j.dd^cu_1^j\wedge...\wedge dd^cu_k^j\wedge T\weak u_0.dd^cu_1\wedge...\wedge dd^cu_k\wedge T,
$$
weakly in the sense of currents.
\end{theorem}

\begin{proof}
See \cite{SA12}.
\end{proof}
\medskip

One of the most important properties of $m$-subharmonic functions is the quasicontinuity. Every $m$-subharmonic function is continuous outside an arbitrarily small open subset. The $m$-Capacity is used to measure the smallness of these sets. 
\begin{definition}\label{def: capacity}
Let $E\subset \Omega$ be a Borel subset. The $m$-capacity of $E$ with respect to $\Omega$ is defined to be
$$
\Ca_m(E,\Omega):=\sup\Big\{\int_E H_m(\f) \ / \ \f\in \SH_m(\Omega) , 0\leq \f\leq 1\Big\}.
$$
\end{definition}
\noindent The $m$-Capacity shares the same elementary properties as the Capacity introduced by Bedford and Taylor.
\begin{prop}\label{prop: properties of m-capacity}
i) $\Ca_m(E_1,\Omega)\leq \Ca_m(E_2,\Omega)$ if $E_1\subset E_2,$

ii) $\Ca_m(E,\Omega)= \lim_{j\to \infty}\Ca_m(E_j,\Omega)$ if $E_j\uparrow E$,

iii) $\Ca_m(E,\Omega)\leq \sum \Ca_m(E_j,\Omega)$  for $E=\cup E_j.$
\end{prop} 

\noindent The following results can be proved by repeating the arguments in \cite{Kol05}.


\begin{theorem}\label{thm: m-quasicontinuity}
Every $m$-subharmonic function $u$ defined in $\Omega $ is quasi-continuous. This means that for any positive number $\epsilon$ one can find an open set $U\subset \Omega$ with $\Ca_m(U,\Omega)<\epsilon $ and such that $u$ restricted to $\Omega\setminus U$ is continuous.
\end{theorem}

\begin{theorem}\label{thm: convergence theorem for increasing sequences} Let  $\{u_k^j\}_{j=1}^{\infty} $ be a locally uniformly bounded sequence of  $m$-subharmonic functions in $\Omega$ for $k=1,2,..., N\leq m$ and let $u_k^j\uparrow u_k\in \SH_m(\Omega)\cap L^{\infty}_{\rm loc}$ almost everywhere as $j\to \infty$ for $k=1,2,...,N.$ Then 
$$
dd^cu_1^j\wedge ...\wedge dd^cu_N^j\wedge \beta^{n-m}\rightharpoonup dd^cu_1\wedge ...\wedge dd^c u_N\wedge \beta^{n-m}. 
$$
\end{theorem}
\begin{theorem}[Integration by parts]\label{theorem: integration by parts} Let  $u,v\in \SH_m(\Omega)\cap \loc(\Omega)$ such that $\lim_{z\to\partial \Omega} u(z)=\lim_{z\to\partial \Omega} v(z)=0.$ Then 
$$
\into u dd^c v \wedge T=\into v dd^c u\wedge T, 
$$
where $T=dd^c \varphi_1\wedge ...\wedge dd^c \varphi_{m-1}\wedge \beta^{n-m}$ with $\varphi_1,...,\varphi_{m-1}\in \SH_m(\Omega)\cap L^{\infty}_{{\rm loc}}(\Omega)$. The equality is understood in the sense that if one of the two terms is finite then so is the other, and they are equal.
\end{theorem}

\begin{theorem}[Maximum principle]\label{theorem: maximum principle} If  $u,v\in \SH_m(\Omega)\cap \loc(\Omega)$ then 
$$
\ind_{\{u>v\}}H_m( \max (u,v))=\ind_{\{u>v\}}H_m(u). 
$$
\end{theorem}

\begin{theorem}[Comparison principle]\label{thm: comparison principle 1} Let $u,v\in \SH_m(\Omega)\cap \loc(\Omega)$ such that 
$$
\liminf_{z\to \partial \Omega}(u(z)-v(z))\geq 0.
$$
 Then
$$
\int_{\{u<v\}}(dd^cv)^m\wedge \beta^{n-m} \leq \int_{\{u<v\}}(dd^cu)^m\wedge \beta^{n-m}.
$$
  \end{theorem}


\begin{definition} For a subset $E$ of a domain $\Omega\subset \C^n$ we define the relative $m$-extremal function of $E$ by
$$
u_{m,E,\Omega}:=\sup \{u\in \SH_m(\Omega) \ / \ u<0,\ \text{and}\ u\leq -1\ \text{on}\ E\}.
$$
\end{definition}
It is easy to see that $\um^*$ is $m$-subharmonic in $\Omega.$ 

\begin{prop}\label{prop: properties of m-extremal functions}  i) If $E_1\subset E_2$ then $u_{E_2}\leq u_{E_1}.$

ii) If $E\subset \Omega_1\subset \Omega_2$ then $u_{m,E,\Omega_2}\leq u_{m,E,\Omega_1}.$

iii) If $K_j\downarrow K,$ with $K_j$ compact in $\Omega$ then $(\lim u_{m,K_j,\Omega}^*)^*=u_{m,K,\Omega}^*.$ 
\end{prop}
As in \cite{Kl91} we have the following formula for the $m$-extremal functions of concentric balls.

\begin{lemma}\label{lem: m-extremal function of the ball} Let $0<r<R$ and set $a=\dfrac{n}{m}>1$. The $m$-extremal function of $B(r)$ with respect to $B(R)$ is given by
$$
u_{m,B(r), B(R)}= \max \Big(\dfrac{R^{2-2a}-\Vert z\Vert^{2-2a}}{r^{2-2a}-R^{2-2a}}, -1\Big).
$$
\end{lemma}


\begin{prop} If $E\Subset \Omega$ then  one has
$
\lim_{z\to w} u_{m,E,\Omega}(z)=0
$
for any $w\in \pO$. 
\end{prop}



\begin{prop}\label{prop: union of balls} Let $K\subset \Omega$ be a compact subset which is the union of closed balls, then $u_K^*=u_K$ is continuous. In particular, if $K\subset \Omega$ is an arbitrary compact set and $\epsilon<{\rm dist}(K, \partial \Omega)$, then $u_{K_{\epsilon}}$  is continuous, where 
$
K_{\epsilon}=\{z\in \Omega \ / \ {\rm dist}(z,K)\leq \epsilon\}. 
$
\end{prop}

\noindent{\bf Sketch of proof}: From Lemma \ref{lem: m-extremal function of the ball} we know that the $m$-extremal function of balls is continuous. Let $K$ be a compact set which is a union of closed balls, $K=\cup B_j$ and let $u$ be its $m$-extremal function. Since $u\leq u_{m,B_j,\Omega}$, it is easy to see that $u^*\equiv-1$ on $K$. The same arguments as in \cite[Proposition 4.5.3]{Kl91} show that $u$ is continuous in $\Omega.$

\begin{definition} The outer $m$-capacity of a Borel set $E\subset \Omega$ is defined by
$$
\Ca_m^*(E,G):=\inf \left\{ \Ca_m(G,\Omega)\  / \ E\subset G, G {\ \rm is \ an \ open\ subset \ of\  } \Omega \right\}.
$$
\end{definition}

\begin{theorem}\label{theorem: m-capacity and m-extremal functions} 
If $E\Subset \Omega$ is a Borel subset then
$$
\Ca^*_m(E,\Omega)=\int_{\Omega}H_m(u_{m,E,\Omega}^*),\ \ {\rm and }\ \Ca^*_m(E,\Omega)=\Ca_m(E,\Omega) \ \ {\rm if }\ E\ {\rm is\ compact}.
$$

\end{theorem}

We compute the $m$-Capacity of the concentric balls.
\begin{examp} For every $0<r<R$ we have
\begin{equation*}\label{example: balls}
\Ca_m(B(r),B(R))= \dfrac{2^n(n-m)}{m.n!(r^{2-2a}-R^{2-2a})^m}.
\end{equation*}
\end{examp}

\begin{definition}\label{def: negligible sets}
Let $\Omega$ be an open set in $\C^n$, and let $\mathcal{U}\subset \SH_m(\Omega)$ be a family of functions which is locally bounded from above. Define 
$$
u(z)=\sup\{v(z) \ / \ v\in \mathcal{U}\}. 
$$
Sets of the form $\mathcal{N}=\{z\in \Omega \ / \ u(z)<u^*(z)\}$
and all their subsets are called $m$-negligible.
\end{definition}


\begin{definition}\label{def: m-polar sets}
A set $E\subset \C^n$ is called $m$-polar if  $E\subset \{v=-\infty\}$ for some $v\in \SH_m(\C^n)$ and $v\not \equiv -\infty.$
\end{definition}

\begin{theorem} \label{theorem: m-negligible and m-polar}
Let $E\Subset \Omega$. Then $E$ is $m$-negligible $\Leftrightarrow$ $E$ is $m$-polar $\Leftrightarrow$ $\Ca^*_m(E,\Omega)=0$.
 \end{theorem}
%
%


\subsection{$m$-polarity of sets with small Hausdorff measure}

In this section, following \cite{Labutin} we give a sufficient condition for a set being $m$-polar using Hausdorff measure. We then give examples of $m$-polar sets ($m<n$) which are not pluripolar. 
 
 \begin{definition}
 A function $h:[0,1)\rightarrow \R^+$ is called measuring function if it is increasing and $\lim_{r \to 0}h(r)=h(0)=0.$ 
 
 For a measuring function $h$, the $h$-Hausdorff measure of $E\subset \C^n$ is defined by (see \cite{Labutin}, \cite{Mattila}, \cite{Ziemer})
 $$
 \Lambda_h(E):=\lim_{\delta \to 0} \left( \inf \sum_{j}h(r_j)\right),
 $$
 where the infimum is taken over all coverings of $E$ by balls $B_j$ of radii $r_j\leq \delta.$
 \end{definition}
 
\begin{theorem}\label{thm:hausdorff measure}
 Let $H(r)=r^{2n-2m}$ ($n>m\geq 1$).  Then every subset $E\subset \C^n$ satisfying $\Lambda_H(E)<+\infty$ is $m$-polar.
 \end{theorem}
\begin{proof}
We can assume that $E\Subset B=B(0,R).$ Suppose that $E$ is not $m$-polar. Then 
 $\Ca^*_{m,B}(E)>0$ and $u:=u^*_{m,E,B}\not \equiv 0.$ Set $E_1:=\{z\in E\ / \ u(z)=-1\}.$ Then by Theorem
  \ref{theorem: m-negligible and m-polar}, $E\setminus E_1$ is $m$-polar and hence 
 $$
 \int_{E\setminus E_1} H_m(u)=0.
 $$
Therefore,  as $H_m(u)$ is a regular Borel measure and $E_1$ is a Borel set, we can find a compact set $K\subset E_1$ such that $\int_K H_m(u)>0.$

We claim that for every bounded open set $\Omega\supset K$, $\Ca_m(K,\Omega)\leq C.\Lambda_H(K)$,
where $C>0$ is a constant independent of $\Omega.$ Indeed, let $\delta:={\rm dist}(K,\Omega)$ and fix $\epsilon \in (0,1)$ such that $\epsilon < \delta/4.$ We cover $K$ by open balls $B(z_j,r_j)$ such that $r_j<\epsilon.$ We may assume that $B(z_j,r_j)\subset \Omega.$ From (\ref{example: balls}) and after simple computations we get 
$$
\Ca_m(B(z_j,r_j),B(z_j,\delta/2))\leq C. r_j^{2n-2m},
$$
where $C>0$ is a constant depending only on $n,m.$
Using this and the monotonicity and subadditivity of the outer $m$-capacity we get 
\begin{eqnarray*}
\Ca^*_m(K,\Omega)&\leq& \sum_{j} \Ca_m^*(B(z_j,r_j),\Omega)\\
&\leq & \sum_{j} \Ca^*_m(B(z_j,r_j), B(z_j,\delta/2))\leq C\sum_{j} H(r_j).
\end{eqnarray*}
Now, the claim follows by taking the infimum over all such coverings and letting $\epsilon \to 0.$ 

For each $\delta>0$ set $\Omega_{\delta}: = \{z\in B\ / \ {\rm dist}(z,K)<\delta\}.$ Since $u$ is continuous at every point on $K$ and $u(x)=-1$ for every $x\in K$, we get
$$
c(\delta):={\rm osc}_{\Omega_{\delta}}(u)\rightarrow 0 \ \ {\rm as} \ \delta \to 0.
$$
For any $z\in \Omega_{\delta}$ we have 
$
0\leq \frac{u(z)+1}{c(\delta)}\leq 1. 
$
Thus, 
\begin{eqnarray*}
\Ca^*_m(K,\Omega_{\delta})&\geq & \Ca_m(K,\Omega_{\delta})\geq \int_K H_m\left(\frac{u+1}{c(\delta)}\right) =  \frac{1}{c(\delta)^m}\int_K H_m(u).
\end{eqnarray*}
We then get $\Ca^*_m(K,\Omega_{\delta})\to+\infty$ as $\delta \to 0$. The Claim yields $\Lambda_H(E)=+\infty.$ 
\end{proof}

\begin{examp} Assume that $1\leq m<n$ and let $E$ be the Cantor set constructed in \cite{Labutin} (see also \cite{AH}, \cite{Carleson}, \cite{Mattila}) and $E^n=E\times E\times...\times E$ ($n$ times). Since $\lim_{r\to 0}H(r)/[\log(1/r)]^{-n}=0$, we can choose the sequence $(\ell_j)$ defining $E$ such that $E^n$ is not pluripolar but $\Lambda_H(E^n)=0$. This implies that $E^n$ is $m$-polar in view of Theorem \ref{thm:hausdorff measure}. 
\end{examp}

\section{Finite energy classes}
In this section we study finite energy classes of $m$-subharmonic functions in $m$-hyperconvex domains. 
They are   generalizations of Cegrell's classes \cite{Ceg98,Ceg04} for plurisubharmonic functions. 

\subsection{Definitions and properties}
In  pluripotential theory one of the most important steps is to regularize singular plurisubharmonic functions. 
It can be easily done locally by convolution with a smooth kernel. 
The following theorem explains how to do it globally in a $m$-hyperconvex domain.
Let $\SH_m^-(\Omega)$ denote  the class of  non-positive functions in $\SH_m(\Omega).$
\begin{theorem}\label{thm: continuous global approximation}
For each $\f\in \SH_m^-(\Omega)$ there  exists a sequence $(\f_j)$ of $m$-sh functions verifying the following conditions:

(i) $\f_j$ is continuous on $\bar{\Omega}$ and $\f_j\equiv 0$ on $\pO;$

(ii) each $H_m(\f_j)$ has compact support,

(iii) $\f_j \downarrow \f$ on $\Omega.$
\end{theorem}
\begin{proof}
If $B$ is a closed ball in $\Omega$ then by Proposition \ref{prop: union of balls}, the $m$-extremal function
 $u_{m,B,\Omega}$ is continuous on $\bar{\Omega}$ 
 and   ${\rm supp} (H_m(u)) \Subset \Omega.$ We can follow the lines in  \cite[Theorem 2.1]{Ceg04}. 
\end{proof}


\begin{definition}
We let $\Ecoo$ denote the class of bounded functions in  $\SH_m^-(\Omega)$ 
such that 
$\underset{z\to \pO}{\lim} \f(z)=0$ and $\int_{\Omega}H_m(\f)<+\infty.$ 

For each $p>0$,  $\Ec_m^p(\Omega)$ denote the class of  functions   $\f\in \SH_m(\Omega)$ 
such that there exists a decreasing sequence  $(\f_j)\subset \Ec_m^0(\Omega)$ satisfying

(i) $\lim_j \f_j=\f,$ in $\Omega$ and

(ii) $\sup_j\int_{\Omega}(-\f_j)^pH_m(\f_j)<+\infty.$ 

If we require moreover that $\sup_j \into H_m(\f_j)<+\infty$ then, by definition,  
 $\f$ belongs to $\Fcpo.$
\end{definition}

\begin{definition}
We let $\Fco$ denote  the  class of functions  $u\in \SH_m^-(\Omega)$ such that there exists 
a sequence  $(u_j)\subset \Ecoo$ decreasing to  $u$ in $\Omega$  and
$$
\sup_j\into H_m(u_j)<+\infty.
$$
\end{definition}

\begin{definition}
We define the $p$-energy ($p>0$) of $\f\in \Ecoo$ by
$$
\E_p(\f):=\int_{\Omega}(-\f)^pH_m(\f).
$$
If $p=1$ we drop the index and denote by $\E(\varphi)=\E_1(\varphi)$.
\end{definition}

We generalize  H\"{o}lder inequality in the following lemma. 
When $m=n$ it is a result of  Persson \cite{Per99}. Our proof uses the same idea.

\begin{lemma}\label{lem: persson}
Let  $u,v_1,...,v_m\in \Ec_m^0(\Omega)$ and $p\geq 1.$ We have
\begin{equation*}
\int_{\Omega}(-u)^pdd^cv_1\wedge...\wedge dd^c v_m\wedge \beta^{n-m}
\leq  D_p(\E_p(u))^{\frac{p}{m+p}}\E_p(v_1)^{\frac{1}{m+p}}...\E_p(v_m)^{\frac{1}{m+p}},
\end{equation*}
here $D_1=1$ and for each $p>1$, $D_p:=p^{p\alpha(p,m)/(p-1)}$, where
$$
\alpha(p,m)=(p+2)\Big(\frac{p+1}{p}\Big)^{m-2}-p-1.
$$
\end{lemma}
\begin{proof} Let
$$
F(u,v_1,...,v_m)=\into(-u)^pdd^cv_1\wedge...\wedge dd^c v_m\wedge \beta^{n-m},\ u,v_1,...,v_m\in \Ec_m^0(\Omega).
$$
Thanks to \cite[Theorem 4.1]{Per99} it suffices to prove that 
\begin{equation}\label{eq: persson 1}
F(u,v,v_1,...,v_{m-1})\leq a(p) F(u,u,v_1,...,v_{m-1})^{\frac{p}{p+1}}F(v,v,v_1,...,v_{m-1})^{\frac{1}{p+1}}, 
\end{equation}
where $a(p)=1$ if $p=1$ and $a(p)=p^{\frac{p}{p-1}}$ if $p>1.$ 
Set $T=dd^c v_1\wedge...\wedge dd^c v_{m-1}\wedge \beta^{n-m}.$ When $p=1,$ (\ref{eq: persson 1}) becomes 
$$
\into (-u)dd^c v\wedge T\leq \Big(\into (-u)dd^c u\wedge T\Big)^{\frac{1}{2}}\Big(\into (-v) dd^c v \wedge T\Big)^{\frac{1}{2}},
$$
which is the Cauchy-Schwarz inequality.  In the case $p>1,$ integrating by parts we get
$$
\into (-u)^pdd^c v\wedge T\leq p\into (-u)^{p-1} (-v) dd^c u\wedge T.
$$
By using H\"{o}lder inequality we obtain
\begin{equation*}
\into (-u)^pdd^c v\wedge T\leq p\Big(\into (-u)^p dd^c u\wedge T\Big)^{\frac{p-1}{p}}
\Big(\into (-v)^pdd^c u\wedge T\Big)^{\frac{1}{p}}.
\end{equation*}
Now, interchanging $u$ and $v$ we get
\begin{equation*}
\into (-v)^pdd^c u\wedge T\leq p\Big(\into (-u)^p dd^c v\wedge T\Big)^{\frac{1}{p}}
\Big(\into (-v)^pdd^c v\wedge T\Big)^{\frac{p-1}{p}}.
\end{equation*}
Combining  the above two inequalities we obtain the result.

\end{proof}
Thanks to Lemma \ref{lem: persson} we can  bound  
$\into (u_0)^p dd^c u_1\wedge...\wedge dd^c u_m\wedge \beta^{n-m}$ by
  $\E_p(u_j), j=0,...,m$ if $p\geq 1.$ To get similar estimates when 
 $p\in (0,1)$ we can follow the lines in \cite{GZ07}: 
\begin{lemma}\label{lem: weighted GZ 1}
Let $u,v\in \Ecoo$ and $0<p<1$. If T is a closed  $m$-positive current of type  
$T=dd^cv_1\wedge...\wedge dd^cv_{m-k}\wedge\beta^{n-m},$ where 
$u_j\in \SH_m(\Omega)\cap L^{\infty}_{\rm loc}$,  then
$$
\into (-u)^p(dd^cv)^k\wedge T\leq 2\into (-u)^p(dd^cu)^k\wedge T+2\into (-v)^p(dd^cv)^k\wedge T.
$$
\end{lemma}
\begin{proof}
The same as in the proof of   \cite[Proposition 2.5]{GZ07}.  
\end{proof}
\begin{prop}\label{prop: weight GZ 2}
Let $0<p<1$. There exists  $C_p>0$ such that  
$$
0 \leq \into (-\f_0)^p dd^c\f_1 \wedge \cdots \wedge dd^c\f_m\wedge\beta^{n-m}
\leq C_p \max_{0 \leq j \leq m} \E_p(\f_j),
$$
for all $0 \geq  \f_0,\ldots,\f_m \in \Ecoo.$
\end{prop}
\begin{proof}
See \cite[Proposition 2.10]{GZ07}
\end{proof}
From Lemma \ref{lem: persson} and Proposition \ref{prop: weight GZ 2} we easily get the following result.
\begin{corollary}\label{cor: sum of functions in ep}
Let $(u_j)$ be a sequence in $\Ecoo$ and  $p>0.$ Assume also that  
 $\sup_j \E_p(u_j)<+\infty.$ Then 
$$
u=\sum_{j=1}^{\infty}2^{-j}u_j \ \text{belongs to } \ \Ecpo.
$$
\end{corollary}
From the above facts, we can prove the convexity the classes $\Ecpo, \Eco$ by the same way as in \cite{Ceg98,Ceg04}.
\begin{theorem}\label{thm: convexite cegrell}
By  $\Ec$ we denote one of the classes $\Ecoo,$  $\Fco$, $\Ecpo$, $\Fcpo$, $p>0.$ They are convex 
and moreover, if   $v\in \Ec$, $u\in \SH_m^-(\Omega)$, $u\geq v$, then  $u\in \Ec.$
\end{theorem}

\subsection{Definition of the complex Hessian operator and basic properties}
In this section  we prove that the complex Hessian operator  $H_m$ is well-defined for functions in 
$\Fco$ and in $\Ecpo, p>0.$  We follow the arguments in \cite{Ceg04}.

As in \cite{Ceg04}, continuous functions in $\Ecoo$ can be considered as
test functions. 
\begin{lemma}\label{lem: test functions}
$\Cc_0^{\infty}(\Omega)\subset \Ecoo\cap \Cc(\Omega) - \Ecoo\cap \Cc(\Omega).$
\end{lemma}

\begin{theorem}\label{thm: convergence E}
Let $u^p\in \Fco, p=1,...,m$ and $(g^p_j)_j\subset \Ecoo$ such that 
$g^p_j\downarrow u^p, \forall p$. Then  the sequence of  measures 
$$
dd^c g^1_j\wedge dd^c g^2_j\wedge...\wedge dd^cg^m_j\wedge \beta^{n-m}
$$
converges weakly to a positive Radon measure which does not depend on the choice of the sequences
 $(g^p_j).$ We then define 
$dd^cu^1\wedge...\wedge dd^cu^m\wedge \beta^{n-m}$ to be this weak limit.
\end{theorem}

\begin{proof}
See \cite[Theorem 4.2]{Ceg04}.
\end{proof}

It is convenient to use the notation $H_m(u_1,...,u_m):=dd^cu^1\wedge...\wedge dd^cu^m\wedge \beta^{n-m}$. When $u_1=...=u_m=u$ we simply write $H_m(u)$ instead of $H_m(u,...,u).$

\begin{definition}
A function $u$ belongs to the class $\Fcao$ if $u\in \Fco$ and $H_m(u)$ vanishes on $m$-polar sets.
\end{definition}
\begin{corollary}\label{cor: def  hes E}
Let $u_1,...,u_m\in \Fco$ and $u_1^j,...,u_m^j$  be sequences of  functions in
$\Ecoo\cap \Cc(\Omega)$ decreasing to $u_1,...,u_m$ respectively such that 
$$
\sup_{j,p} \into H_m(u_j^p)<+\infty.
$$ 
Then  for each  $\f\in \Ecoo\cap \Cc(\Omega)$ we have
$$
\lim_{j\to+\infty} \into \f dd^c u_1^j\wedge...\wedge dd^c u_m^j\wedge\beta^{n-m}=\into  \f dd^c u_1\wedge...\wedge dd^c u_m\wedge \beta^{n-m}.
$$
\end{corollary}
\begin{proof}
It is clear that
\begin{equation}\label{eq: class E 1}
\sup_j \into dd^c u_1^j\wedge...\wedge dd^c u_m^j\wedge\beta^{n-m}<+\infty.
\end{equation}
Fix $\epsilon>0$ small enough and consider $\f_{\epsilon}=\max(\f,-\epsilon).$ 
The function $\f-\f_\epsilon$ is continuous and compactly supported in  $\Omega.$ 
It follows from Theorem \ref{thm: convergence E} that
$$
\lim_{j\to+\infty} \into (\f-\f_\epsilon) dd^c u_1^j\wedge...
\wedge dd^c u_m^j\wedge\beta^{n-m}=\into  (\f-\f_\epsilon) 
dd^c u_1\wedge...\wedge dd^c u_m\wedge \beta^{n-m}.
$$
Observe also that  $\vert \f_\epsilon\vert \leq \epsilon$. By using (\ref{eq: class E 1}), we get the result.
\end{proof}
\begin{corollary}\label{cor: strong convergence}
 Assume that $(u_j)\subset \Ecoo$ decreases to $u$ such that 
 $$
 \sup_{j} \into H_m(u_j)<+\infty
 $$
 Then for every  $h\in \Ecoo$ we have the weak convergence 
 $$
 h H_m(u_j)\weak h H_m(u).
 $$
\end{corollary}
\begin{proof}
 For every test function $\chi$ the function $h\chi$ is upper semicontinuous. Thus, 
 $$
\liminf_{j\to+\infty} \into (-h)\chi H_m(u_j)\geq \into (-h)\chi H_m(u).
 $$
 Let $\Theta$ be any cluster point of this the sequence $(-h) H_m(u_j)$. From the above inequality we infer that 
 $\Theta \geq (-h)H_m(u).$ Moreover, it follows from Corollary \ref{cor: def hes E}  that the sequence $\into (-h)H_m(u_j)$ increases to $\into (-h) H_m(u).$ This implies that the total mass of $\Theta$ is less than or equal to the total mass of $(-h) H_m(u)$ and hence these measures are equal.
\end{proof}

\begin{theorem}\label{thm: convergence Ep 1}
Let  $u_1,...,u_m\in \Ecpo, \ p>0$ and $(g^j_1), (g^j_2),...,(g^j_m)\subset \Ecoo$ 
be such that  $g^j_k\downarrow u_k, \forall k=1,...,m$ and
$$
\sup_{j,k} \E_p(g^j_k)<+\infty.
$$
Then the sequence of measures  $dd^c g^j_1\wedge dd^c g^j_2
\wedge...\wedge dd^cg^j_m\wedge \beta^{n-m}$
 converges  weakly to  a positive Radon measure which does not depend on the choice of the sequences
 $(g^j_k).$  We then define  $dd^cu_1\wedge...\wedge dd^cu_m\wedge \beta^{n-m}$
 to be this weak limit.
\end{theorem}

\begin{proof}
Since the problem is local we can assume that $u_1,...,u_m\in \Fcpo.$ Indeed, let $K$ be a compact  subset of $\Omega.$ For each $j\in \N, k=1,...,m$ consider
$$
h_k^j:=\sup\{u\in \SH_m(\Omega) \ / \ u\leq g_k^j\ {\rm on}\ K\}.
$$
Then by using a standard balayage argument we see that $H_m(h_k^j)$ is supported in $K$. It follows that  $h_k^j$ decreases to $v_k\in \Fcpo$. Observe also that $v_k=u_k$ on $K.$
\medskip
 
Now, fix $h\in \Ecoo.$ Then
$$
\into h dd^c g^j_1\wedge dd^c g^j_2\wedge...\wedge dd^c g^j_m\wedge \beta^{n-m}
$$
is  decreasing to a finite number.
Thus the limit
$\lim_j \into h dd^c g^j_1\wedge dd^c g^j_2\wedge...\wedge dd^c g^j_m\wedge \beta^{n-m}$
 exists for every  $h\in \Ecoo.$ In view of Lemma \ref{lem: test functions}, this implies the weak convergence of the sequence  
$$
dd^c g^j_1\wedge dd^c g^j_2\wedge...\wedge dd^c g^j_m\wedge \beta^{n-m} .
$$ 
To prove the last  statement it suffices to follow the lines in \cite[Theorem 4.2]{Ceg04}.
\end{proof}
\subsection{Integration by parts and the comparison principle}
In this section we prove that integration by parts is valid in the classes $\Ecpo, p>0$ and $\Fco.$
Following a classical argument of Cegrell in \cite{Ceg04} we prove that the comparison principle holds in $\Ecpo$ for $0<p\leq 1$. To prove the comparison principle  in $\Ecpo, p>1$  and in $\Fco$ we need strong convergence results which will be established in the next section. 
\medskip

From Theorem \ref{thm: convergence E} and Corollary  \ref{cor: def  hes E} we prove the integration
 by parts formula for functions in $\Fco$. 
\begin{theorem}\label{thm: int by parts F} Integration by parts is allowed in $\Fco$, more precisely
$$
\into udd^cv\wedge T=\into vdd^cu\wedge T,
$$
where $u,v, \f_1,...,\f_{m-1}\in \Fco$ and $T=dd^c \f_1\wedge ...\wedge dd^c \f_{m-1}\wedge \beta^{n-m}$ and the equality means that if one of the two terms are finite then they are equal. 
\end{theorem}
\begin{proof}
Let $u_j ,v_j,\f_1^j,..,\f_{m-1}^j$ be sequences in  $\Ecoo\cap \Cc(\Omega)$ 
decreasing to   $u, v, \f_1,...,\f_{m-1}$ respectively such that their total mass are uniformly bounded:
$$
\sup_j\into  dd^c v_j\wedge T_j<+\infty,\ \  \sup_j\into  dd^c u_j\wedge T_j<+\infty,
$$
where $T_j=dd^c \f_1^j\wedge ...\wedge dd^c \f_{m-1}^j\wedge \beta^{n-m}.$  
Theorem \ref{thm: convergence E} gives us that $dd^c u_j\wedge T_j\weak dd^c u\wedge T.$
For each fixed  $k\in \N$ and any $j>k$ we have
$$
\into v_kdd^cu_k\wedge T_k\geq \into v_kdd^c u_j\wedge T_j\geq \into v_jdd^cu_j\wedge T_j.
$$ 
We then infer that  the sequence of real numbers  $\into v_jdd^cu_j\wedge T$  decreases to some $a\in \Rinf$. 
By letting $j\to +\infty$ and using Corollary \ref{cor: def  hes E}  we get 
$$
\into v_k dd^c u\wedge T\geq a,
$$ 
from which we obtain  $\into v dd^c u\wedge T\geq a.$  For each  fixed  $k$ we also have
\begin{eqnarray*}
\into v dd^c u\wedge T&\leq &\into v_k dd^c u\wedge T=\lim_{j\to+\infty}\into v_k dd^c u_j\wedge T_j\\
&\leq &\into v_k dd^c u_k\wedge T_k.
\end{eqnarray*}
This implies that $\into vdd^cu\wedge T=a,$ from which the result follows.
\end{proof}

Let $u\in \Ecpo, p>0.$ It is clear form the definition that the $H_m(u)$ does not charge $m$-polar sets.  One expects that $H_m(u)$ is dominated by the Capacity. The following results tell us more about that. 

\begin{lemma}\label{lem: capacity Ep p>1}
Let $U$  be  an open subset of $\Omega$ and $\f\in \Ecoo, \ p\geq 1.$  Then
$$
\int_U H_m(\f)\leq \Ca_m(U)^{\frac{p}{p+m}}\E_p(\f)^{\frac{m}{p+m}}.
$$
\end{lemma}
\begin{proof}
We can suppose that $U$ is relatively compact in $\Omega.$ Denote by  $u=u_{m,U,\Omega}$  the $m$-extremal function of
 $U$ in $\Omega.$ Then  $u\in \Ecoo$ and $u=-1$ in $U.$  From Lemma \ref{lem: persson} we have
\begin{eqnarray*}
\int_U H_m(\f)&\leq & \into (-u)^pH_m(\f)\leq \E_p(u)^{\frac{p}{m+p}}\E_p(\f)^{\frac{m}{m+p}}\\
&\leq & \Big(\into H_m(u)\Big)^{\frac{p}{m+p}}\E_p(\f)^{\frac{m}{p+m}}= \Ca_m(U)^{\frac{p}{p+m}}\E_p(\f)^{\frac{m}{p+m}}.
\end{eqnarray*}
\end{proof}
\begin{lemma}\label{lem: capacity Ep 0<p<1}
Let $U\subset \Omega$ be an open subset and  $\f\in \Ecoo, \ 0<p\leq 1.$  Then for each   $\epsilon>0$  small enough  we have 
$$
\int_U H_m(\f)\leq 2(\Ca_m(U))^{1-m\epsilon}+2\Ca_m(U)^{p\epsilon} . \E_p(\f).
$$
\end{lemma}
\begin{proof}
Without loss of generality we can assume that  $U\Subset \Omega.$
Let  $u$ be  the $m$-extremal function of  $U$ with respect to $\Omega.$ 
Put $a=\Ca_m(U)=\into H_m(u).$ If $a=0$, we are done.  Thus, we can assume that 
$a>0.$ By applying  Lemma \ref{lem: weighted GZ 1} we obtain 
\begin{eqnarray*}
\int_U H_m(\f)&\leq& a^{p\epsilon}\into (-u/a^{\epsilon})^p H_m(\f)\\
&\leq& 2a^{p\epsilon}\E_p(u/a^{\epsilon})+ 2a^{p\epsilon}\E_p(\f)\\
&\leq& 2a^{1-m\epsilon}+2a^{p\epsilon}\E_p(\f).
\end{eqnarray*}
\end{proof}
The following result is the so-called {\it maximum principle}.
\begin{theorem}\label{thm: principe du maximum Ep}
Let  $u_1,...,u_m\in \Ecpo, \ p>0$ and $v\in \SH_m(\Omega)$. Then
$$
\ind_{A} H_m(u_1,...,u_m)=\ind_{A} H_m(\max(u_1,v),...,\max(u_m,v)),
$$
where $A=\cap_{j=1}^m \{u_j>v\}$ and $H_m(u_1,...,u_m)= dd^c u_1\wedge...\wedge dd^c u_m\wedge \beta^{n-m}.$
\end{theorem}
\begin{proof}
Let $(u_1^j,...u_m^j)$ be  sequences in $\Ecoo$  decreasing to $u_1,...,u_m$ respectively as in the definition of  $\Ecpo.$ We can assume that they are continuous in $\Omega.$ Set $v_k^j:=\max(u_k^j,v), k=1,...,m.$ Then since the set $A_j:=\cap_{k=1}^m \{u_k^j>v\}$ is open, we get
\begin{equation}\label{eq: comparison principle in Ep 1}
\ind_{A_j} H_m(u_1^j,...,u_m^j)=\ind_{A_j} H_m(v_1^j,...,v_m^j).
\end{equation}
Set $u^j:=\min(u_1^j,...,u_m^j)$ and $u:= \min (u_1,...,u_m).$ Consider $\p_j:=\max(u^j-v,0)$.  Then  $\p_j\downarrow \p:=\max(u-v,0)$, 
all of them are  quasi-continuous. 

Fix  $\delta>0$ and set  $g_j:=\frac{\p_j}{\p_j+\delta},\ g=\frac{\p}{\p+\delta}.$ By multiplying 
  (\ref{eq: comparison principle in Ep 1}) with $g_j$ we obtain
\begin{equation}\label{eq: comparison principle in Ep 2}
g_jH_m(u_1^j,...,u_m^j)=g_j H_m(v_1^j,...,v_m^j).
\end{equation}

Now, let $\chi\in \Cc_0^{\infty}(\Omega)$ be a test function and fix $\epsilon>0.$ By Theorem \ref{thm: m-quasicontinuity}, there exists an open subset 
 $U\subset \Omega$ such that  $\Ca_m(U)<\epsilon,$ and there exist $\f_j,\f$  continuous functions in 
 $\Omega$ which coincide with $\p_j,\p$ respectively on $K:=\Omega\setminus U.$  The monotone convergence 
   $\p_j\downarrow \p$ implies that $\f_j$ converges uniformly to $\f$ on $K\cap {\rm Supp}\chi,$ which in turn  
   implies the uniform convergence of
    $h_j=\frac{\f_j}{\f_j+\delta}$ on $K\cap {\rm Supp} \chi$  to $h=\frac{\f}{\f+\delta}.$ 

In the next arguments,  we let $C$  denote a positive constant which does not depend on
$j, \epsilon.$  Since  $g_j, h_j$ are uniformly bounded , Lemma  
\ref{lem: capacity Ep p>1} and Lemma \ref{lem: capacity Ep 0<p<1} give us
\begin{equation}\label{eq: 12Nov 1}
\Big\vert\into \chi g_jH_m(u_1^j,...,u_m^j) - \into \chi h_jH_m(u_1^j,...,u_m^j) \Big\vert\leq C.\int_U H_m(u_1^j,...,u_m^j)\leq C.\epsilon^q,
\end{equation}
where $q$ is some positive constant. The last inequality follows since $H_m(u_1,...,u_m)\leq H_m(u_1+...+u_m)$ and since $\Ecpo$ is convex.
We also obtain
\begin{eqnarray}\label{eq: 12Nov 2}
\Big\vert\into \chi g\, H_m(u_1,...,u_m) - \into \chi h\, H_m(u_1,...,u_m) \Big\vert &\leq & C.\int_U H_m(u_1+...+u_m)\nonumber \\
&\leq & C.\epsilon^q.
\end{eqnarray}
Moreover, since $h$ is continuous on $\Omega$ and $H_m(u_1^j,...,u_m^j)\weak H_m(u_1,...,u_m)$, we get
\begin{equation*}
\lim_{j\to+\infty}\into \chi. h (H_m(u_1^j,...,u_m^j)-H_m(u_1,...,u_m))=0.
\end{equation*}
Hence, we obtain
\begin{eqnarray*}\label{eq: 12Nov 3}
&&\limsup_{j\to+\infty}\Big\vert\into \chi h_j\, H_m(u_1^j,...,u_m^j) - \into \chi h\, H_m(u_1,...,u_m) \Big\vert \\
& \leq &
 \limsup_{j\to+\infty} \into \chi.\vert h_j-h\vert H_m(u_1^j,...,u_m^j).
\end{eqnarray*}
Since $h_j$ converges uniformly to $h$ on $K\cap \rm{supp}\chi$, we have
\begin{eqnarray*}\label{eq: 12Nov 4}
\into \chi.\vert h_j-h\vert H_m(u_1^j,...,u_m^j)&= & \int_U \chi.\vert h_j-h\vert H_m(u_1^j,...,u_m^j) \\
&+& \int_K \chi.\vert h_j-h\vert H_m(u_1^j,...,u_m^j)\\
&\leq & C. \int_U H_m(u_1^j,...,u_m^j)\\
&+&\Vert h_j-h\Vert_{L^{\infty}(K\cap \rm{supp} \chi)}\into \chi H_m(u_1^j,...,u_m^j).\nonumber
\end{eqnarray*}
From the two inequalities above we get
\begin{eqnarray}\label{eq: 12Nov 5}
\limsup_{j\to+\infty}\Big\vert\into \chi h_j\, H_m(u_1^j,...,u_m^j) - \into \chi h\, H_m(u_1,...,u_m) \Big\vert & \leq &
 C. \epsilon^{q},
\end{eqnarray}
where $q$ is some positive constant. From  (\ref{eq: 12Nov 1}), (\ref{eq: 12Nov 2}) and  (\ref{eq: 12Nov 5}), we see that   
$$
\limsup_j \Big\vert\into \chi g_jH_m(u_1^j,...,u_m^j) - \into \chi g H_m(u_1,...,u_m) \Big\vert \leq C.\epsilon^q.
$$
We then see that $g_jH_m(u_1^j,...,u_m^j)\rightharpoonup g H_m(u_1,...,u_m)$. In the same way, we get 
$$
g_jH_m(v_1^j,...,v_m^j)\rightharpoonup g H_m(v_1,...,v_m),
$$
and hence  $ g H_m(u_1,...,u_m)=g H_m(v_1,...,v_m).$  The result follows by letting   $\delta$ go to zero.
\end{proof}

We now return to the integration by parts formula in the class $\Ecpo$, $p>0$. We first need the following convergence result.
\begin{corollary}\label{cor: def hes Ep}
Let $u_1,...,u_m\in \Ecpo, p>0$ and $u_1^j,...,u_m^j$ be sequences of functions in
$\Ecoo$  decreasing to  $u_1,...,u_m$ respectively such that 
$$
\sup_{j,k} \into (-u_k^j)^pH_m(u_k^j)<+\infty.
$$ 
Then for each  $\f\in \Ecoo$ we have
$$
\lim_{j\to+\infty} \into \f dd^c u_1^j\wedge...\wedge dd^c u_m^j\wedge\beta^{n-m}=\into  \f dd^c u_1\wedge...\wedge dd^c u_m\wedge \beta^{n-m}.
$$
\end{corollary}
\begin{proof}
Without loss of generality we can assume that the right hand side is finite. Since if it is $-\infty$ then the equality is obvious. 
We can also assume that $-1\leq \varphi\leq 0.$
We will use a truncation argument. For each $k, j\in \N$ set $u_{l,k}^{j}:=\max(u_l^j, k\psi)$, and  $u_{l,k}:=\max(u_l,k\psi)$,  $l=1,...,m.$ Here we set $\psi:=-(-\varphi)^q$, where $q=\min(1,1/p).$

We claim that, for any $k$,
$$
\lim_{j\to+\infty} \into (-\varphi) H_m(u_{1,k}^j,...,u_{m,k}^j) = \into (-\varphi) H_m(u_{1,k},...,u_{m,k}).
$$
Indeed, the inequality "$\geq$" follows from the fact that the sequence of Hessian measures converges  and $-\varphi$ is lower semi continuous. Moreover, it follows from Theorem \ref{theorem: integration by parts} that we can integrate by parts in the right hand side, which implies the inequality $"\leq"$. Thus, the Claim is proved.

Thus, it is enough to prove that 

$$
\Big\vert \into (-\varphi) H_m(u_{1,k}^j,...,u_{m,k}^j) -\into (-\varphi) H_m(u_{1}^j,...,u_{m}^j) \Big\vert \leq \epsilon(k),
$$
where $0<\epsilon(k)\to 0$ as $k\to+\infty$ and (of course) $\epsilon(k)$ does not depend on $j.$ If we can prove it then the same estimate holds for the limit functions $u_1,...,u_m$ and we are done. \\
In the following arguments we use $C_1,C_2,...$ to denote positive constants that do not depend on $j,k.$ By Theorem \ref{thm:  principe du maximum Ep} and since $\varphi$ is bounded it suffices to estimate
$$
\int_{\{u_l^j\leq k\psi\}} (-\varphi) H_m(u_{1,k}^j,...,u_{m,k}^j), \ \ l=1,...,m.
$$
But we can bound this term by using Lemma \ref{lem: persson} (for $p\geq 1$), Proposition \ref{prop: weight GZ 2} (for $0<p<1$) and  the fact that the $p$-energy of these functions are uniformly bounded: 
\begin{eqnarray*}
\int_{\{u_l^j\leq k\psi\}} (-\varphi) H_m(u_{1,k}^j,...,u_{m,k}^j)&\leq& \frac{1}{k^p} \into (-u_{l,k}^j)^p H_m(u_{1,k}^j,...,u_{m,k}^j)
\leq  \frac{C_1}{k^p}.
\end{eqnarray*}
In the last step the constant $C_1$ depends on the $p$-energy of $u_{l,k}^j, l=1,...,m$. Note also that  $u_{l,k}^j \geq u_l^j$.
Now, for functions $u,v\in \Ecoo$ such that $u\leq v$ we always have $E_p(v)\leq C_2.E_p(u)$ (where $C_2$ does not depend on $u,v$). To see this we can use integration by parts (if $p<1$) or use Lemma \ref{lem: persson}. This explains why the constant $C_1$ in the above estimate does not depend on $k.$ Thus, the proof is complete. 
\end{proof}

\begin{theorem}\label{thm: integration par parties Ep p>0}
Integration by parts is allowed in  $\Ecpo, p>0.$  More precisely, assume that $u,v\in \Ecpo$ and $T$
is  a closed $m$-positive
current of type  $T=dd^c \f_1\wedge...dd^c\f_{m-1}\wedge \beta^{n-m}$, where
 $\f_j\in \Ecpo, \forall j.$ Then  
$$
\int_{\Omega}udd^c v\wedge T=\int_{\Omega}vdd^c u\wedge T,
$$
where the equality means that if one of the two terms is finite then so is the other and they are equal.
\end{theorem}
\begin{proof}
 Thanks to Corollary \ref{cor: def hes Ep} the same arguments as in the proof of Theorem \ref{thm: int by parts F} can be used here.
\end{proof}
\begin{theorem}\label{thm: principe de comparaison Ep}
Let $u,v\in \Ecpo, p>0$ (or $\Fco$) such that  $u\leq v$ on $\Omega.$ Then
$$
\into H_m(u)\geq \into H_m(v).
$$
\end{theorem}
\begin{proof}
Let  $(u_j), (v_j)$ be two sequences in  $\Ecoo$ decreasing to  $u, v$ as in the definition of $\Ecpo.$ Fix $h\in \Ecoo\cap \Cc(\Omega).$  We can suppose that 
 $u_j\leq v_j,\ \forall j.$  Integrating by parts we get
$$
\into (-h)H_m(v_j)\leq \into (-h)H_m(u_j).
$$
Corollary \ref{cor: def hes E} and Corollary \ref{cor: def hes Ep} then yield
$$
\lim_{j\to \infty}  \into (-h)H_m(v_j)=\into (-h) H_m(v),\ \ \ {\rm and} \ \ \ \lim_{j\to \infty}  \into (-h)H_m(u_j)=\into (-h) H_m(u).
$$
Combining them we obtain
$$
\into (-h)H_m(v)\leq \into (-h)H_m(u).
$$
The result follows by letting $h$ decrease to $-1.$
\end{proof}
\begin{theorem}\label{thm: convergence Ep 2}
If $u\in \Ecpo$, $p>0$, then $\E_p(u):=\into (-u)^pH_m(u)<+\infty.$
If $(u_0^j),...,(u_m^j)$ are sequences in $\Ecoo$ decreasing to $u_0,...,u_m\in \Ecpo$ respectively  then
$$
\into (-u_0^j)dd^c u_1^j\wedge...\wedge dd^cu_m^j\wedge\beta^{n-m}
\nearrow \into (-u_0)dd^c u_1\wedge...\wedge dd^cu_m\wedge \beta^{n-m}.
$$
\end{theorem}
\begin{proof}
Let $(u_j)$ be a sequence in $\Ecoo$ decreasing to $u$ and having  uniformly bounded $p$-energy. Then
$$
\into (-u)H_m(u)\leq \liminf_{j\to+\infty} \into(-u_j)H_m(u_j) <+\infty.
$$
We now prove the second statement.  We can assume that the sequences have uniformly bounded $p$-energy.
 It follows from Theorem  \ref{thm: convergence Ep 1} that
$$
T_j:=dd^cu_1^j\wedge...\wedge dd^cu_m^j\wedge \beta^{n-m}\rightharpoonup 
T:=dd^cu_1\wedge...\wedge dd^cu_m\wedge \beta^{n-m}.
$$
Furthermore since $(-u_0^j)\uparrow (-u_0)$ and since all of them are lower semicontinuous, we have
$$
\liminf_j\into (-u_0^j)T_j\geq \into (-u_0^j)T.
$$
Thus, it suffices to prove that
$$
\into (-u_0)T_j\leq \into (-u_0)T, \forall j.
$$
But it can be easily seen by integrating by parts thanks to Theorem \ref{thm: integration par parties Ep p>0}. The proof is thus complete.
\end{proof}
\begin{theorem}\label{thm: principe de comparaison  Ep 2}
If $0<p\leq 1$ and $u,v\in \Ecpo$  then
$$
\int_{\{u>v\}} H_m(u)\leq \int_{\{u>v\}}H_m(v).
$$
\end{theorem}
\begin{proof}
Fix $h\in \Ecoo\cap \Cc(\Omega).$ The measure  $H_m(v)$ 
does not charge $m$-polar sets. We can easily show that for almost every $r,$
$$
\int_{\{v=ru\}}(-h)H_m(v)=0 .
$$
This allows us  to restrict ourself to the case $\int_{\{u=v\}}(-h)H_m(v)=0.$
From Theorem \ref{thm: principe du maximum Ep},  we get
$$
\ind_{\{u>v\}}H_m(u)=\ind_{\{u>v\}} H_m(\max(u,v)), \ \text{and}\  \ind_{\{u<v\}}H_m(v)=\ind_{\{u<v\}} H_m(\max(u,v)).
$$
Furthermore, as in the proof of Theorem \ref{thm: principe de comparaison Ep}, we can prove that
$$
\into (-h)H_m(\max(u,v))\leq \into (-h)H_m(u).
$$
From this we get
\begin{eqnarray*}
\int_{\{u>v\}}(-h)H_m(u) &=& \int_{\{u>v\}}(-h)H_m(\max(u,v))\\
&\leq &\into (-h)H_m(\max(u,v))+\int_{\{u< v\}}hH_m(\max(u,v))\\
&\leq &\into (-h)H_m(v)+\int_{\{u<v\}}hH_m(v)=\int_{\{u>v\}}(-h)H_m(v).
\end{eqnarray*}
The above arguments hold since all terms are finite. This is no longer true if $p>1.$
Now, letting  $h\downarrow -1$ we obtain the result. 
\end{proof}
\begin{rmk}\label{rmk: comparion 1} We proved in the above arguments that  
$$
\int_{\{u>v\}}(-h)H_m(v)\leq \int_{\{u>v\}}(-h)H_m(u)
$$
if $u,v\in \Ecpo, 0< p\leq 1$ and $h\in \Ecoo\cap \Cc(\Omega).$  
Thanks to the regularization theorem   (Theorem \ref{thm: continuous global approximation})
it also holds for every  $h\in \SH_m^-(\Omega).$
\end{rmk}
\begin{theorem}\label{thm: principe de comparaison 3}
Let  $u,v\in \Ecpo$, $0<p\leq 1$, such that  $H_m(u)\geq H_m(v).$ 
Then  $u\leq v$ in $\Omega$.
\end{theorem}
\begin{proof}
See \cite[Theorem 4.5]{Ceg98}. 
\end{proof}


\section{The variational approach}
In this section we use a variational method to solve the equation 
$
H_m(u)=\mu,
$
where $\mu$  is a positive Radon measure. We characterize the range of  $H_m(u)$ when $u$ runs in $\Ecpo.$ 

Our results  are direct generalizations of the classical case
of plurisubharmonic functions (see \cite{ACC10}, \cite{Ceg98,Ceg04}). The variational approach for the complex
Monge-Amp\`ere equation was first introduced in \cite{BBGZ13}.

\subsection{The energy functional}

We recall some useful results obtained from previous sections.
For $\varphi\in \Ecto$, we define its energy by $\E(\varphi)=\into (-\varphi)H_m(\varphi).$
\begin{itemize}
\item  If $0\geq u_j\downarrow u$ and $u\in \Ecto$,  then  by Theorem \ref{thm: convergence Ep 2}, we have
 $\E(u_j)\uparrow \E(u).$
\item  If $u,v\in \Ecto$ and $u\leq v$ then $\E(u)\geq \E(v).$
\end{itemize}
\begin{lemma}\label{lem: compact E1}
(i) If $(u_j)\subset \Ecto$ then $(\sup_j u_j)^{*}\in \Ecto.$ 

(ii) If $(u_j)\in \Ecto$ such that  $\sup_j \E(u_j)<+\infty$ and $u_j\downarrow u,$ 
then $u\in \Ecto.$

(iii) For each $C>0$,  $\Ec_m^{1,C}:=\{u \in \Ecto \ / \  \E (u)\leq C \}$ is convex and compact in 
 $\SH_m(\Omega).$
\end{lemma}
\begin{proof}
(i) Let $(\f_j)$ be  a sequence of continuous functions in  $\Ecoo$ decreasing to 
 $\f:=(\sup_j u_j)^{*}.$  Since $u_j\leq \f_j$,  we have $\sup_j \E (\f_j)<+\infty,$ 
 which implies that $\f\in \Ecto.$

(ii) Let   $(\f_j)$ be a sequence in $\Ecoo\cap \Cc(\Omega)$ decreasing to $u.$ 
Set $\p_j:=\max(u_j,\f_j)$.  Then $\p_j \in \Ecto, \forall j$ and  $\E(\p_j)\leq \E(u_j)$. 
Thus, $u\in \Ecto.$

(iii) Let $(u_j)$ be a sequence in $\Ec_m^{1,C}$.  Since $\sup_{j} \E(u_j)<+\infty,$ $(u_j)$  can not go 
uniformly to  $-\infty$ in $\Omega.$  Thus, there exists a subsequence (still denoted by $(u_j)$)  converging 
to  $u\in \SH_m(\Omega)$ in $L^1_{\rm loc}(\Omega)$. Set
$$
\f_j:= (\sup_{k\geq j} u_k)^{*} \in \Ecto, \forall j.
$$
Then $\f_j\downarrow u$ and $\sup_j \E(\f_j)\leq C.$  In view of  (ii), we have  $u\in \Ecto,$ and since 
$(-\f_j)\uparrow (-u)$, all of them being lower semicontinuous  we get, for each fixed $k\in \N$, 
$$
C\geq \liminf_{j\to\infty} \into (-\f_j) H_m(\f_j) \geq  \liminf_{j\to\infty} \into (-\f_k) H_m(\f_j)\geq \into (-\f_k) H_m(u).
$$
By monotone convergence Theorem we see that $E_1(u)\leq C.$
This means $u\in \Ec_m^{1,C}.$
\end{proof}
\begin{lemma}\label{lem: change measure}
Let $\mu$ be a positive Radon measure  in $\Omega$ such that  $\mu(\Omega)<+\infty$ and
$\mu$ does not charge $m$-polar sets. Let  $(u_j)$ be a sequence in  $\SH_m^-(\Omega)$
which converges in $L^1_{\rm loc}$ to $u\in \SH_m^-(\Omega).$ 
 If $\sup_j \into (-u_j)^2d\mu <+\infty$ then  $\into u_jd\mu\to \into ud\mu.$
\end{lemma}
\begin{proof}
Since  $\into u_jd\mu$ is bounded it suffices to prove that  every cluster point is $\into u d\mu.$  
Without loss of generality we can assume that $\into u_j d\mu $ converges.  Since the sequence 
  $u_j$ is bounded in $L^2(\mu)$,  one can apply Banach-Saks theorem  to extract a subsequence
   (still denoted by  $u_j$) such that
$$
\f_N:=\frac{1}{N}\sum_{j=1}^Nu_j
$$
converges in $L^2(\mu)$ and  $\mu$-almost everywhere to $\f.$ Observe also that
 $\f_N\to u$ in $L^1_{\rm loc}$.  For each   $j\in\N$  set
$$
\p_j:=(\sup_{k\geq j}\f_k)^{*}.
$$
Then $\p_j\downarrow u$ in $\Omega.$  But $\mu$ does not charge the $m$-polar set
$
\{(\sup_{k\geq j}\f_k)^{*}>\sup_{k\geq j}\f_k\}.
$
We thus get
$\p_j=\sup_{k\geq j}\f_k$ $\mu$-almost everywhere. Therefore, 
 $\p_j$ converges to $\f$ $\mu$-almost everywhere hence $u=\f$ $\mu$-almost everywhere. This yields
$$
\lim_j \into u_jd\mu = \lim_j \into \f_jd\mu =\into u d\mu.
$$
\end{proof}
\begin{lemma}\label{lem: e1 usc}
The functional  $\E: \Ecto\rightarrow \R$ is lower semicontinuous.
\end{lemma}
\begin{proof}
Suppose that $u, u_j\in \Ecto$ and $u_j$ converges to $u$ in $L^1_{\rm loc}(\Omega).$   
We are  to prove that $\liminf _j \E(u_j)\geq \E(u).$ For each $j\in\N$, the function
$$
\f_j:=(\sup_{k\geq j}u_k)^*
$$
belongs to $\Ecto$ and $\f_j\downarrow u.$  Hence  $\E(\f_j)\uparrow \E(u)$. 
We also have  $\E(u_j)\geq \E(\f_j)$ from which the result follows.
\end{proof}
\begin{definition}
Let $\mu$ be a positive Radon measure in $\Omega.$  The functional  $\Fcmu: \Ecto\rightarrow \R$  is defined by
$$
\Fcmu(u)=\frac{1}{m+1} \E(u) +\Lcm(u),
$$
where $\Lcm(u)= \into u d\mu.$
We say that $\Fcmu$ is proper  (with respect to $\E$) if $\Fcmu\to +\infty$ whenever $\E\to +\infty.$   
\end{definition}

\begin{definition}
For each $p>0$, let $\Mcc_p$ denote the set of all positive Radon measures $\mu$ in $\Omega$ such that $\Ecpo\subset L^p(\Omega,\mu).$
\end{definition}

\begin{prop}\label{prop: equiv ep}
Let $\mu$ be a positive Radon measure on  $\Omega$ and $p>0.$  Then $\mu\in \Mcc_p$  if and only if  there exists a positive constant $C=C(p)>0$ such that
$$
\into (-u)^pd\mu \leq C. \E_p(u)^{\frac{p}{m+p}},\ \forall u\in \Ecpo.
$$
\end{prop}
\begin{proof}
The "if" statement is  evident. To prove the "only if",  suppose by contradiction that
  $\mu\in \Mcc_p$  and 
  there exists a sequence  $(u_j)\subset \Ecpo$ such that 
$$
\into (-u_j)^pd\mu \geq 4^{jp}\E_p(u_j)^{\frac{p}{m+p}}.
$$ 
For simplicity we can assume that  $\E_p(u_j)=1, \forall j.$  
By Corollary   \ref{cor: sum of functions in ep},
$v=\sum_{j=1}^{\infty} 2^{-j}v_j$ belongs to $\Ecpo.$ But
$$
\into (-v)^pd\mu \geq \into(-2^{-j}v_j)^pd\mu \geq 2^{jp}\to +\infty,
$$
which contradicts $\Ecpo \subset L^p(\mu).$
\end{proof}
\begin{rmk}\label{rmk: equiv Ep}
If $u,v\in \Ecpo$ and $(u_j), (v_j)\subset \Ecoo$  decrease to
$u,v$ respectively, then by Lemma \ref{lem: persson} and Proposition
 \ref{prop: weight GZ 2} we have
$$
\into (-u)^pH_m(v)\leq \liminf_j \into (-u_j)^pH_m(v_j)<+\infty.
$$
Thus, $\Ecpo\subset L^p(H_m(v))$  and by Proposition \ref{prop: equiv ep} 
there exists  $C_v>0$ such that
$$
\into (-u)^pH_m(v)\leq C_v\E_p(u)^{\frac{p}{p+m}}, \ \forall u\in \Ecpo.
$$
It is not clear how to obtain this inequality directly by using  H\"{o}lder inequality.
\end{rmk}

\begin{lemma} \label{lem: convexity of Fcmu}
If $u,v\in \Ecto$ then
$
\E(u+v)^{\frac{1}{m+1}} \leq \, \E(u)^{\frac{1}{m+1}} + \, \E(v)^{\frac{1}{m+1}}\, .
$
Moreover, if $\mu\in \Mcc_1$ then $\Fcmu$ is convex and proper.
\end{lemma}
\begin{proof} 
It follows from Lemma  \ref{lem: persson} that
$$
\E(u+v) \leq  \E(u)^{\frac{1}{m+1}}\, \E(u+v)^{\frac{m}{m+1}} +
 \E(v)^{\frac{1}{m+1}}\, \E(u+v)^{\frac{m}{m+1}}\ , 
$$
which implies that  $\E^{\frac{1}{m+1}}$ is convex since it is homogeneous of degree $1$. So,  $\E$ is also convex. 
If $\mu$ belongs to  $\mathcal M_1$, there exists $A>0$ such that
$$
\|u\|_{L^1(\mu)}\leq A. \E(u)^{\frac{1}{1+m}}\, ,\qquad \text{for every } u\in \Ecto\, .
$$
We thus obtain
$$
\Fcmu(u_j)=\frac {1}{m+1}\E(u_j)-\|u_j\|\Lmu\geq \frac {1}{m+1}\E(u_j)-A.\E(u_j)^{\frac {1}{m+1}}\to \infty.
$$
\end{proof}
Let $u:\Omega\rightarrow \R\cup \{-\infty\}$ be an upper semicontinuous function.
Suppose that there exists  $w\in \Ecto$ such that  $w\leq u.$ We define the projection 
of  $u$ on $\Ecto$ by
$$
P(u):=\sup\{v\in \Ecto\ / \ v\leq u\}.
$$

Using  a standard balayage argument, we obtain the following result.
\begin{lemma}\label{lem: balayage}
Let $u:\Omega\rightarrow \R$ be a continuous function. Suppose that there exists 
 $w\in \Ecto$ such that $w\leq u.$ Then
\begin{equation*}\label{eq: balayage 1}
\int_{\{P(u)<u\}} H_m(P(u))=0.
\end{equation*} 
\end{lemma}
\begin{proof}
Without loss of generality we can assume that $w$  is bounded.  From Choquet's lemma,
there exists an increasing sequence  $(u_j)\subset \Ecto\cap L^{\infty}(\Omega)$ such that 
$$
(\lim_j u_j)^*=P(u).
$$
Let $x_0\in \{P(u)<u\}.$  Since $u$ is continuous, there exists $\epsilon>0, r>0$ such that 
$$
P(u)(x)<u(x_0)-\epsilon<u(x),\ \ \forall x\in B=B(x_0,r).
$$
For each fixed $j$, by approximating $u_j\vert_{\partial B}$ from above by a sequence of continuous functions
on  $\partial B$ and by using \cite[Theorem 2.10]{DK11}, we can find a function $\f_j\in \SH_m(B)$ such 
that $\f_j=u_j$ on $\partial B$ and $H_m(\f_j)=0$ in $B.$   The comparison principle gives us that $\f_j\geq u_j$ in $B$. 
The function $\psi_j$, defined by $\psi_j=\f_j$ in  $B$ and $\psi_j=u_j$ in $\Omega\setminus B,$  belongs to
 $\Ecto\cap L^{\infty}(\Omega).$
For each  $x\in \partial B$ we have
$\f_j(x)=u_j(x)\leq P(u)(x)\leq u(x_0)-\epsilon.$ 
It then follows that   $\f_j\leq u(x_0)-\epsilon$ in $B$ since $u(x_0)-\epsilon$  is a constant and 
 $\f_j$ is $m$-sh. Hence, $u_j\leq \psi_j\leq u$ in $\Omega.$ This implies that
$$
(\lim \psi_j)^*=P(u).
$$
It follows from Theorem \ref{thm: convergence theorem for increasing sequences} that
 $H_m(\psi_j)\weak H_m(P(u)).$ Therefore, 
 $$
 H_m(P(u))(B)\leq \liminf_{j\to+\infty} H_m(\psi_j)(B)=0,
 $$
 from which the result follows.
\end{proof}
\begin{lemma}\label{lem: comparison principle} Let $u,v\in \Ecto$ and suppose that $v$ is continuous.
For each  $t<0,$ we define 
$$
h_t=\frac{P(u+tv)-tv-u}{t}.
$$
Then for each  $0\leq k \leq m,$ 
\begin{equation}\label{eq: cp_1}
\lim_{t\nearrow 0}\into h_t(dd^cu)^k\wedge (dd^cP(u+tv))^{m-k}\wedge\beta^{n-m} = 0\, .
\end{equation}
In particular,
\begin{equation}\label{eq: cp2}
\lim_{t\nearrow 0}\into \frac{P(u+tv)  -u}{t}(dd^cu)^k\wedge (dd^cP(u+tv))^{m-k}\wedge\beta^{n-m} = \into vH_m(u)\, .
\end{equation}
\end{lemma}
\begin{proof}
An easy computation shows that $h_t$ is decreasing in  $t$ and $0\leq h_t \leq -v.$
 For each fixed $s<0$ we have
\begin{eqnarray*}
&&\lim_{t\nearrow 0}\into h_t(dd^cu)^k\wedge (dd^cP(u+tv))^{m-k}\wedge\beta^{n-m}\\
&&\leq \lim_{t\nearrow 0}\into h_s(dd^cu)^k\wedge (dd^cP(u+tv))^{m-k}\wedge\beta^{n-m}\\ 
&&=\into h_s(dd^cu)^m\wedge\beta^{n-m} \leq \int\limits_{\{P(u+sv) - sv<u \}}(-v) (dd^cu)^m\wedge\beta^{n-m}.
\end{eqnarray*}
Let  $u_k\in \Ecoo\cap C(\Omega)$ be a sequence decreasing to  $u$ such that
$$
\int\limits_{\{P(u+sv)-sv<u \}} (-v)(dd^cu)^m\wedge\beta^{n-m} \leq 2\int\limits_{\{P(u_k+sv)-sv < u \}}
(-v)(dd^cu)^m\wedge\beta^{n-m}\, .
$$
Taking into account Remark \ref{rmk: comparion 1}  and Lemma \ref{lem: balayage} we can conclude that
\begin{eqnarray*}
&&\int\limits_{\{P(u_k+sv) - sv < u \}}(-v) (dd^cu)^m\wedge\beta^{n-m}\\
&&\leq \int\limits_{\{P(u_k+sv) - sv < u_k \}}(-v) (dd^c(P(u_k+sv) - sv))^m\wedge\beta^{n-m}\\
&& \leq -s M \to 0\, , \qquad \text{as } s \to 0.
\end{eqnarray*}
Here,  $M$  is a positive constant which depends only on
 $m$, $\|v\|$,  and $\into v\, H_m (u+v)$.  
 Equality (\ref{eq: cp2}) follows from equality (\ref{eq: cp_1}). The proof is thus complete.
\end{proof}
\begin{lemma}\label{lem: derivative}
Let $u,v\in \Ecto$, and assume that  $v$ is continuous. Then
$$
\frac{d}{dt}\Big\vert_{t=0}\E(P(u+tv))=\into(-v)H_m(u).
$$
\end{lemma}
\begin{proof}
If $t>0$, $P(u+tv)=u+tv.$  It is easy to see that
$$
\frac{d}{dt}\Big\vert_{t=0^+}\E(P(u+tv))=\into(-v)H_m(u).
$$

To compute the left-derivative observe that
$$
\begin{aligned}
&\frac 1t\left (\into(-P(u+tv))(dd^cP(u+tv))^m\wedge\beta^{n-m}-\into(-u)(dd^cu)^m\wedge\beta^{n-m}\right )\\
&=\sum_{k=0}^m\into \frac{u-P(u+tv)}{t}(dd^cu)^k\wedge (dd^cP(u+tv))^{m-k}\wedge\beta^{n-m}.
\end{aligned}
$$
It suffices to  apply Lemma \ref{lem: comparison principle}.
\end{proof}
%
%
%
%
%
%
\subsection{Resolution}
In this section we use the variational formula established above to solve the equation
 $H_m(u)=\mu$ in finite energy calsses of Cegrell type,  where $\mu$ is a positive Radon measure.  Our main results represented in the introduction follow from these theorems. The following lemma is important for the sequel.
\begin{lemma}\label{lem: preparing lemma} Let $\mu$ be a positive Radon measure such that $\Fcmu$ is proper and lower semicontinuous on $\Ecto$. Then there exists $\varphi \in \Ecto$ such that
$$
\Fcmu(\varphi)=\inf_{\psi\in \Ecto} \Fcmu(\psi).
$$
\end{lemma}
\begin{proof} 
Let  $(\varphi_j)\subset \Ecto$ be such that 
$$
\lim_j \Fcmu(\varphi_j)=\inf_{\psi\in \Ecto} \Fcmu(\psi)\leq 0.
$$
From the properness of the  functional $\Fcmu$, we obtain
  $\sup_j \E(\varphi_j)<+\infty.$  It follows that the sequence $(\varphi_j)$ forms a compact subset of $\Ecto.$ 
 Hence there exists a subsequence  (still denoted by $(\varphi_j)$)  such that $\varphi_j$  converges to $\varphi$ in
 $L^1_{\rm loc}(\Omega)$. Since $\Fcmu$ is lower semicontinuous we have   
$$
\liminf_{j\to\infty}\Fcmu(\varphi_j)\geq \Fcmu(\varphi).
$$
We then deduce that  $\varphi$  is a minimum point of  $\Fcmu$ on $\Ecto.$ 
\end{proof}
We now prove a Dirichlet principle.
\begin{theorem}\label{thm: Dirichlet principle}
 Let $\varphi\in \Ecto$ and $\mu\in \Mcc_1$. Then
$
H_m(\varphi)=\mu \Leftrightarrow \Fcmu (\varphi)=\inf_{\psi\in \Ecto} \Fcmu (\psi).
$
\end{theorem}

\begin{proof}
Assume first that $H_m(\varphi)=\mu.$  Let $\psi\in \Ecto.$ By Lemma \ref{lem: persson} and H\"{o}lder inequality we get
$$
\into (-\psi)H_m(\varphi)\leq \E(\psi)^{1/(1+m)}. \E(\varphi)^{m/(1+m)} \leq \frac{1}{m+1}\E(\psi)+\frac{m}{m+1}\E(\varphi).
$$
We then easily obtain  $\Fcmu(\psi)\geq \Fcmu(\varphi).$

Now, assume that $\varphi$ minimizes $\Fcmu$ on $\Ecto.$ Let $\psi$ be a continuous function in $\Ecto$ and consider the function $g(t)=\E(P(\varphi+t.\psi))+\Lcm(\varphi+t.\psi), \ \ t\in \R.$  Since $P(\varphi+t\psi)\leq \varphi+t\psi$, we have that 
$$
g(t)\geq \Fcmu(P(\varphi+t\psi))\geq \Fcmu(\varphi)=g(0), \ \forall t.
$$
It follows that $g$ attains its minimum at $t=0$, hence $g'(0)=0$. Since $\mu\in \Mcc_1$, $\Lcm$ is finite on $\Ecto$ which implies that 
$$
\frac{d}{dt}\Lcm(\varphi+t.\psi)=\Lcm(\varphi).
$$
This coupled with Lemma \ref{lem: derivative} yields
$$
\into \psi H_m(\varphi)=\into \psi d\mu.
$$
The test function $\psi$ is taken arbitrarily, so it follows that $\mu=H_m(\varphi).$
\end{proof}

\begin{theorem}\label{thm: first main theorem}
Let $\mu$ be a positive Radon measure such that $\Ecto\subset L^1(\Omega,\mu).$ Then there exists a unique  $u\in \Ecto$ such that $H_m(u)=\mu.$
\end{theorem}

\begin{proof}
The uniqueness follows from the comparison principle.
\medskip

We prove the existence.  Suppose first that  $\mu$  has compact support
 $K\Subset\Omega$, and let $h_K:=h_{m,K,\Omega}^{*}$ denote the 
  $m$-extremal function of $K$ with respect to  $\Omega.$ Set
$$
\mathcal{M}=\left\{\nu\geq 0\ / \  {\rm supp} (\nu)\subset K,\;
\int_{\Omega}(-\varphi)^2 d\nu\leq C.\E(\varphi)^{\frac{2}{m+1}}\text{ for every } \varphi\in\Ecto\right\},
$$
where  $C$ is a fixed constant such that $C>2 \E(h_K)^{\frac{m-1}{m+1}}.$
For each  compact  $L\subset K$ , we have $h_K\leq h_L$. We deduce that  $\E(h_L)\leq \E(h_K).$ 
Therefore, for every  $\varphi\in\Ecto$, we have
\begin{eqnarray*}
\into (-\varphi)^2H_m(h_L)&\leq& 2\|h_L\|\into (-\varphi)(dd^c\varphi)\wedge (dd^c h_L)^{m-1}\wedge\omega^{n-m}\\ 
&\leq& 2\left(\int_{\Omega} (-\varphi)H_m(\f)\right)^{\frac{2}{m+1}}
\left(\int_{\Omega} (-h_L)H_m(h_L)\right)^{\frac{m-1}{m+1}}\\
&\leq& C. \E(\varphi)^{\frac{2}{m+1}}.
\end{eqnarray*}
This implies that  $H_m(h_L)\in\mathcal{M}$ for every  compact $L\subset K$.
\medskip

Put $T=\sup\{\nu(\Omega)\ / \ \nu\in\mathcal{M}\}.$  We claim that $T<+\infty.$
In fact, since $\Omega$ is $m$-hyperconvex, there exists 
 $h\in\SH_m^-(\Omega)\cap \Cc(\bar{\Omega})$ such that 
$
K\Subset \{h<-1\}\Subset \Omega.
$
For each  $\nu\in \mathcal{M}$, we have
$$
\nu(K)\leq \int_K (-h)d\nu \leq C. \E(h)^{\frac{2}{m+1}},
$$
from which the claim follows.
\medskip

Fix $\nu_0\in \mathcal{M}$ such that $\nu_0(\Omega)>0$. Let $\mathcal{M}'$ denote the set of all probability measures $\nu$ in $\Omega$ supported in $K$ such that
$$
\into(-\varphi)^2 d\nu\leq \left(\frac{C}{T} + \frac{C}{\nu_0(\Omega)}\right)
\E(\varphi)^{\frac{2}{m+1}} , \ \ \forall \varphi\in\Ecto \, .
$$
Then, for each   $\nu\in \mathcal M$ and  $\varphi\in\Ecto,$
\begin{eqnarray*}
\into (-\varphi)^2\frac{(T-\nu(\Omega))d\nu_0 +\nu_0(\Omega)d\nu}{T\nu_0(\Omega)} &\leq&
\frac{T-\nu(\Omega)}{T\nu_0(\Omega)}\into (-\varphi)^2d\nu_0+\frac1{T}\into(-\varphi)^2d\nu\\ 
&\leq &
\left(C\frac{T-\nu(\Omega)}{T\nu_0(\Omega)} + \frac{C}{T}\right)\E(\varphi)^{\frac{2}{m+1}}\\
&\leq&
\left(\frac{C}{\nu_0(\Omega)} + \frac{C}{T}\right)\E(\varphi)^{\frac{2}{m+1}}.
\end{eqnarray*}
From this we infer that
$$
\frac{(T-\nu(\Omega))\nu_0 +\nu_0(\Omega)\nu}{T\nu_0(\Omega)}\in\mathcal{M}^\prime, \text{ for every } \nu\in\mathcal{M}\, .
$$
We conclude that $\mathcal{M}^\prime$ is (non empty)
 convex and weakly compact in the space of probability  measures.
 It follows from a generalized Radon-Nykodim Theorem  \cite{Rai69} that 
 there exists a positive measure $\nu\in\mathcal{M}^\prime$ and a positive  function $f\in L^1(\nu)$ 
 such that  $\mu=f\,d\nu+\nu_s$, where  $\nu_s$ is orthogonal to $\mathcal{M}^\prime$. 
 Observe also that every measures orthogonal to $\mathcal{M}^\prime$ is supported in some $m$-polar set
 since $H_m(h_L)\in \mathcal M$ for each $L\Subset K.$ 
 We then deduce that  $\nu_s\equiv 0$ since $\mu$ does not charge $m$-polar sets. 
 \medskip
 
From Lemma \ref{lem: change measure}, Lemma \ref{lem: convexity of Fcmu} we see that for each $\lambda\in \mathcal{M}^\prime,$ the functional $\Fc_{\lambda}$ is proper and lower semicontinuous. 
For each $j\in \N$ set $\mu_j=\min(f,j)\nu$. Then   $\mathcal{L}_{\mu_j}$ is also continuous on $\Ecto$ and $\Fc_{\mu_j}$ is proper  since $\mu_j\leq j.\nu.$  
Therefore, by Lemma \ref{lem: preparing lemma} and Theorem \ref{thm: Dirichlet principle}, there exists   $u_j\in\Ecto$ such that $H_m(u_j)=\mu_j$. 
It is clear from the comparison principle that $\{u_j\}$ decreases to a  function $u\in \Ecto$  which solves  $H_m(u)=\mu$.
\medskip

It remains to treat the case when $\mu$ does not have compact support. 
Let  $\{K_j\}$ be an exhaustive sequence of compact subsets of $\Omega$ and consider
$\mu_j=\chi_{K_j}d\mu.$  Let $u_j\in \Ecto$ solve $H_m(u_j)= \mu_j.$ Observe also that $(u_j)$ decreases to $u\in \SH_m^-(\Omega).$ 
 It suffices to prove that $\sup_j \E(u_j)<+\infty$.  Since $\mu\in \Mcc_1$, we have
$$
 \E(u_j)=\into (-u_j)H_m(u_j)=\int_{K_j} (-u_j)d\mu\leq \into (-u_j)d\mu \leq A.\E(u_j)^{\frac{1}{m+1}}.
$$
This implies that $\E(u_j)$ is uniformly bounded, hence  $u$ belongs to $\Ecto$ and the result follows. 
\end{proof}
\section{Some applications}
\begin{lemma}\label{lem: preparing lemma 2}
Let $\mu$ be a positive Radon measure having finite mass $\mu(\Omega)<+\infty.$
Assume that  $\mu\leq H_m(\p)$, where $\p$ is a bounded $m$-sh function in $\Omega.$
Then there exists a unique function $\f\in \Ecoo$ such that $\mu=H_m(\f).$
\end{lemma}
\begin{proof}
Without loss of generality, we can assume that $-1\leq \p\leq 0.$ Consider $h_j=\max(\p,jh)$, 
where $h\in \Ecoo$ is  an exhaustion function of $\Omega.$ Let
$ A_j:=\{z\in \Omega \ / \  jh<-1\}.$
From Theorem  \ref{thm: first main theorem},  there exists $(\f_j)_j\subset \Ecoo$ such that
 $H_m(\f_j)=\ind_{A_j}\mu, \ \forall j.$ Thus,
$$
0\geq \f_j\geq h_j\geq \p,\ \text{and}\ \f_j\downarrow \f\in \Ecoo.
$$
\end{proof}
Using this lemma we can prove the comparison principle for the classes $\Ecpo$ with $p>1.$ 
\begin{theorem}\label{thm: principe de comparaison  Ep 3}
If $p> 1$ and $u,v\in \Ecpo$  then
$$
\int_{\{u>v\}} H_m(u)\leq \int_{\{u>v\}}H_m(v).
$$
\end{theorem}
\begin{proof}
 Fix $0\not\equiv h\in \Ecoo\cap \Cc(\Omega)$. Assume first that $v$ is bounded and vanishes on the boundary. Let $K_j$ be an exhaustion sequence of compact subsets of $\Omega.$ Use Lemma \ref{lem: preparing lemma 2} to solve $H_m(v_j)=\ind_{K_j}H_m(v).$ Then $v_j\downarrow v.$ Now, the arguments of the proof of Theorem \ref{thm: principe de comparaison  Ep 3} can be  applied to yield
 $$
 \int_{\{u>v_j\}} (-h) H_m(u) \leq \int_{\{u>v_j\}} (-h) H_m(v_j) =\int_{\{u>v_j\}\cap K_j} (-h) H_m(v)  .
 $$
 Letting $j\to+\infty$ we get  
 $$
 \int_{\{u>v\}} (-h) H_m(u) \leq \int_{\{u>v\}} (-h) H_m(v) . 
 $$
 It remains to remove the assumption on $v.$ For each $k\in \N$ set 
 $$
 \varphi_k:=\max\left(v, -k(-h)^{1/p}\right) .
 $$
 Since $\varphi_k$ is bounded and vanishes on $\partial \Omega$, by the above arguments we get
  $$
 \int_{\{u>\varphi_k\}} (-h) H_m(u) \leq \int_{\{u>\varphi_k\}} (-h) H_m(\varphi_k) . 
 $$
 Set $A_k:=\{v>-k(-h)^{1/p}\}$ and $B_k:=\{v\leq -k(-h)^{1/p}\}.$ On $B_k$ we have $(-h)\leq \frac{(-\varphi_k)^p}{k^p}.$ It then follows that
 $$
 \int_{B_k} (-h)H_m(\varphi_k) \leq \frac{1}{k^p}E_p(\varphi_k)\leq  \frac{C}{k^p}E_p(v),
 $$
 where $C>0$ does not depend on $k.$ \\
 It follows from Theorem \ref{thm: principe du maximum Ep} that $H_m(\varphi_k)=H_m(v)$ on $A_k$. We thus get
 $$
 \int_{\{u>\varphi_k\}} (-h) H_m(u) \leq \int_{\{u>\varphi_k\}\cap A_k} (-h) H_m(v) + \frac{C}{k^p}E_p(v) . 
 $$
It suffices now to let $k\to+\infty.$
\end{proof}
Now we prove a decomposition theorem of Cegrell type.  
\begin{theorem}\label{thm: Cegrell's decomposition}
Let $\mu$ be  a positive measure in $\Omega$ which does not charge $m$-polar sets.  Then there exists 
  $\f\in \Ecoo$ and $0\leq f\in L^1_{\rm loc}(H_m(\f))$  such that $\mu=f.H_m(\f).$
\end{theorem}
\begin{proof}
We first assume that $\mu$ has compact support. By applying Theorem \ref{thm: first main theorem} 
we can find  $u\in\Ecto$ and $0\leq f\in L^1(H_m(u))$  such that $\mu=f.H_m(u),$ and 
$\text{supp}(H_m(u))\Subset \Omega.$ Consider
$$
\p=(-u)^{-1}\in \SH_m(\Omega)\cap L^{\infty}_{\rm loc}(\Omega).
$$
Then  $(-u)^{-2m}H_m(u)\leq H_m(\p).$  Since   $H_m(u)$ has compact support in   $\Omega,$ 
we can modify  $\p$ in a neighborhood of $\pO$ such that $\p\in \Ecoo.$  It follows from Lemma
\ref{lem: preparing lemma 2} that 
$$
(-u)^{-2m}H_m(u)=H_m(\f), \ \f\in \Ecoo.
$$ 
This gives us $\mu=f(-u)^{2m}.H_m(\f).$

It remains to consider the case  $\mu$  does not have compact support. Let $(K_j)$ be an exhaustive 
 sequence of  compact subsets of  $\Omega.$  From previous arguments  there exists $u_j\in \Ecoo$ 
 and $f_j\in L^1(H_m(u_j))$  such that $\ind_{K_j}\mu=f_jH_m(u_j).$  Take  a sequence of positive numbers
   $(a_j)$ satisfying   $\f:=\sum_{j=1}^{\infty}a_ju_j\in \Ecoo.$  The measure
   $\mu$ is absolutely continuous with respect to $H_m(\f).$ Thus,  
$$
\mu=gH_m(\f)\ \text{and}\ g\in L^1_{\rm loc}(H_m(\f)).
$$
\end{proof}

\begin{theorem}\label{thm: solution Ep}
Let  $\mu$ be a positive Radon measure on $\Omega$ such that 
 $\Ecpo\subset L^p(\mu), p>0.$ Then there exists a unique 
  $\f\in \Ecpo$ such that $H_m(\f)=\mu.$
\end{theorem}
\begin{proof}
The uniqueness follows from the comparison principle. Let us prove the existence result. 
Since  $\mu$  does not charge $m$-polar sets, applying the decomposition theorem 
 (Theorem \ref{thm: Cegrell's decomposition}) we get
$$
\mu=fH_m(u),\ \ u\in \Ecoo, \ 0\leq f\in L^1_{\rm loc}(H_m(u)).
$$
For each  $j$, use Lemma \ref{lem: preparing lemma 2} to find  $\f_j\in \Ecoo$ such that
$$
H_m(\f_j)=\min(f,j)H_m(u).
$$
By Proposition \ref{prop: equiv ep}, $\sup_j \E_p(\f_j)<+\infty.$ Thus, the comparison principle 
gives us that $\f_j\downarrow \f\in \Ecpo$  which solves 
 $H_m(\f)=\mu.$ 
\end{proof}

\section{Proof of the main results}
\subsection{Proof of Theorem 1}
 \begin{prop}\label{prop: convergence Ep 3}
Let $u,v\in \Ecpo, \ p>0.$  There exist two sequences  $(u_j), (v_j)\subset \Ecoo$  decreasing to
$u, v$ respectively such that  
$$
\lim_{j\to+\infty} \into (-u_j)^pH_m(v_j)=\into (-u)^pH_m(v).
$$
In particular, if $\f\in \Ecpo$ then there exists  $(\f_j)\subset \Ecoo$ decreasing to   $\f$  such that 
$$
\E_p(\f_j)\to \E_p(\f).
$$
\end{prop}
\begin{proof}
Let $(u_j)$  be a sequence in  $\Ecoo$  decreasing to  $u$  such that 
 $\sup_j \into (-u_j)^pH_m(u_j)<+\infty.$ Since  $H_m(v)$ vanishes on $m$-polar sets
 Theorem \ref{thm: Cegrell's decomposition} gives
$$
H_m(v)=fH_m(\p),\ \ \p\in \Ecoo, \ 0\leq f\in L^1_{\rm loc}(H_m(\p)).
$$
For each $j$, use Lemma \ref{lem: preparing lemma 2} to find  $v_j\in \Ecoo$ such that
$$
H_m(v_j)=\min(f,j)H_m(\p).
$$
By the comparison principle $v_j\downarrow \f\in \Ecpo$  which solves $H_m(\f)=H_m(v).$
It implies that $\f\equiv v.$  We then have
$$
\into (-u)^pH_m(v)=\lim_j \into (-u_j)^p\min(f,j)H_m(\p)=\lim_j\into (-u_j)^pH_m(v_j).
$$
\end{proof}

Theorem 1 is a consequence of the following result.

\begin{theorem}\label{thm: main 2}
 Let $\mu$ be a positive Radon measure in $\Omega$ and $p>0.$ Then we have
$$
\mu=H_m(\varphi) \  {\rm with}\  \varphi\in \Ecpo\Leftrightarrow  \Ecpo \subset L^p(\Omega,\mu).
$$
\end{theorem}
\begin{proof}
 The implication $\Leftarrow$ has been proved in Theorem \ref{thm: first main theorem} and Theorem \ref{thm: solution Ep}. Now, assume that $\mu=H_m(\varphi)$ with $\varphi\in \Ecpo$. Let  $\psi$ be another function in $\Ecpo$. By Proposition \ref{prop: convergence Ep 3} there exist sequences $(\varphi_j), (\psi_j)$ in $\Ecoo$ having uniformly bounded energy such that
 $$
 \into (-\psi)^pH_m(\varphi)=\lim_j\into (-\psi_j)^pH_m(\varphi_j),\ \ \lim_j\E_p(\varphi_j) =\E_p(\varphi), \ \lim_j\E_p(\psi_j)=\E_p(\psi).
 $$
 It suffices now to apply Lemma \ref{lem: persson}.
\end{proof}

\subsection{Proof of Theorem 2}
\begin{proof} 
We first prove the existence result. 
Since $\mu$ does not charge $m$-polar sets the decomposition theorem yields
$$
\mu=fH_m(u),\ \ u\in \Ecoo, \ 0\leq f\in L^1_{\rm loc}(H_m(u)).
$$
For each $j$ use Lemma \ref{lem: preparing lemma 2} to find $\f_j\in \Ecoo$ such that
$$
H_m(\f_j)=\min(f,j)H_m(u).
$$
Besides, $\sup_j \into H_m(\f_j)\leq \mu(\Omega)<+\infty.$ Thus, 
 $\f_j\downarrow \f\in \Fco$ in view of the comparison principle. The 
 limit function $\f$ solves   $H_m(\f)=\mu$ as required.
 
The uniqueness  can be proved by the same ways as in   \cite[Lemma 5.14]{Ceg04}. 
Assume that $\psi\in \Fco$ solves $H_m(\psi)=\mu.$ We are to prove that $\f=\psi.$ Let $(K_j)$ be an exhaustive  sequence of compact subsets of $\Omega$ such that  $h_j=h_{m,K_j,\Omega}$ is continuous.  For each $j$, the function $\psi_j:=\max(\psi,j.h_j)$ belongs to $\Ecoo$, and $\psi_j\downarrow \psi.$ Set $d_j:=\frac{\psi_j}{j}-h_j=\max(\frac{\psi}{j}-h_j,0).$ Then $d_j\leq \ind_{\{\p>j.h_j\}}$ and $1-d_j\downarrow 0.$ For $s>j$, by the comparison principle we get
\begin{eqnarray*}
 0\leq d_j H_m(\max(\p,s.h_j)) &\leq & \ind_{\{\p>j.h_j\}} H_m(\max(\p,s.h_j))\\
 &=&  \ind_{\{\p>j.h_j\}} H_m(\max(\p,j.h_j))\\
 &=& H_m(\max(\p, s.h_j)).
 \end{eqnarray*}
 Letting $s$ tend to $+\infty$  and using Corollary \ref{cor: strong convergence} we get 
 \begin{equation}\label{eq: uniqueness 1}
 d_j. H_m(\p)\leq \ind_{\{\p>j.h_j\}} H_m(\max(\p,j.h_j)) \leq H_m(\p).
 \end{equation}
 Recall that from the first part we have
 $$
\mu=fH_m(u),\ \ u\in \Ecoo, \ 0\leq f\in L^1_{\rm loc}(H_m(u)).
$$
and $H_m(\f_p)=\min(f,p)H_m(u)$ for each $p\in \N.$ For each $p,j$ we can find $v_j^p \in \Ecoo$ such that
$$
H_m(v_j^p)=(1-d_j)H_m(\f_p).
$$ 
Using (\ref{eq: uniqueness 1}) we get 
\begin{eqnarray}\label{eq: uniqueness 2}
H_m(\f_p)&=& d_j.H_m(\f_p)+(1-d_j)H_m(\f_p)\\
&\leq &  d_j H_m(\psi)+(1-d_j)H_m(\f_p)\nonumber\\
&\leq & \ind_{\{\psi>jh_j\}} H_m(\psi_j)+H_m(v_j^p)\nonumber\\
&\leq & H_m(\psi_j)+H_m(v_j^p).\nonumber
\end{eqnarray}
This couped with the comparison principle yield  $\f_p\geq v_j^p+\psi_j$. Letting $p\to+\infty$ we obtain
$\f\geq v_j+\psi_j$, where $v_j\in \Fco$ solves $H_m(v_j)=(1-d_j)H_m(\f).$ Since $H_m(\f)$ does not charge 
$m$-polar sets, by monotone convergence theorem the total mass of $H_m(v_j)$ goes to $0$ as $j\to+\infty.$
This implies that $v_j$ increases to $0$ and hence $\f\geq \psi.$

Now, we prove that $\f\leq \psi.$ Let $\psi_j, t_j \in \Ecoo$ such that  $H_m(w_j)=d_j H_m(\psi_j)$ and $H_m(t_j)=(1-d_j) H_m(\psi_j).$ Since $H_m(\f_p)$ increases to $H_m(\f)$, the comparison principle can be applied for $\f$ and $w_j$ which implies that $w_j\geq \f.$ But, applying again the comparison principle for $t_j+w_j$ and $\psi_j$ we get $t_j+w_j\leq \psi_j.$ Furthermore, the total mass of $H_m(t_j)$ can be estimated as follows
\begin{eqnarray*}
\into H_m(t_j)&=&\into H_m(\psi_j)-\into d_j H_m(\psi_j)\\
&\leq & \into H_m(\psi)-\into d_j^2 H_m(\psi)\\
&\leq &2 \into (1-d_j)H_m(\psi) \to 0.
\end{eqnarray*}
This implies that $t_j$ converges in $m$-capacity to $0.$ Indeed, for every $\epsilon>0$ and $m$-subharmonic
function $-1\leq \theta\leq 0$, by the comparison principle we have
\begin{eqnarray*}
\epsilon ^m\int_{\{t_j< -\epsilon\}} H_m(\theta)&\leq &\int_{\{t_j< -\epsilon \theta\}} H_m(\epsilon \theta)\\
&\leq & \int_{\{t_j< -\epsilon \theta\}} H_m(t_j)\leq \into H_m(t_j)\to 0.
\end{eqnarray*}
Thus, we can deduce that $\f\leq \psi$ which implies the equality. 
 \end{proof}

 We prove in the above uniqueness theorem that every $\p\in \Fc^a_m(\Omega)$ can be approximated from above by a sequence  $\p_j\in \Ecoo$ such that $H_m(\p_j)$ increases to $H_m(\psi).$ This type of convergence is strong enough to prove the comparison principle for the class $\Fc^a_m(\Omega).$ We thus get

\begin{theorem}
The comparison principle is valid for functions in $\Fc^a_m(\Omega).$
\end{theorem}

\section{Examples}
\begin{lemma}\label{lem: capacity}
If $\f\in \Ecpo$, $p>0$ then 
$
\Ca_m(\f<-t)\leq C.\E_p(\f).\frac{1}{t^{m+p}},
$
where $C>0$ is a constant depending only on $m.$
\end{lemma}
\begin{proof}
Without loss of generality we can assume that $\f\in \Ecoo.$ 
Fix $u\in \SH^-_m(\Omega)$ such that $-1\leq u\leq 0.$ Observe that, for any $t>0,$
$$
(\f<-2t)\subset (\f<tu-t)\subset (\f<-t).
$$
Thus, by the comparison principle (Theorem \ref{thm: comparison principle 1}) we have
\begin{eqnarray*}
\int_{\{\f<-2t\}}H_m(u)\leq \frac{1}{t^m}\int_{\{\f<tu-t\}}H_m(tu-t)\leq \frac{1}{t^m}\int_{\{\f<tu-t\}}H_m(\f)\\
\leq \frac{1}{t^m}\int_{\{\f<-t\}}H_m(\f)\leq\frac{1}{t^{m+p}}\into (-\f)^pH_m(\f).
\end{eqnarray*}
\end{proof}
\begin{prop}\label{prop: measure with density in Lp}
Let $\mu=fdV,$ where $0\leq f\in L^p(\Omega,dV), \frac{n}{m}>p>1.$ Then 
$$
\mu=H_m(\f), \ \f\in \Fc_m^q(\Omega), \ \forall q<\frac{nm(p-1)}{n-mp}.
$$
\end{prop}
\begin{proof}
Fix $0<r<n/(n-m).$ By H\"{o}lder's inequality and   \cite[Proposition 2.1]{DK11}, there exists $C>0$ depending only on $p,r, \Vert f\Vert_p$ such that
\begin{equation}\label{eq: application 1}
\mu(K)\leq C.{\rm Vol} (K)^{\frac{p-1}{p}}\leq C. \Ca_m(K)^{\frac{r(p-1)}{p}}.
\end{equation}
Take $0<q<\frac{nm(p-1)}{n-mp}$ and $u\in \Ec_m^q(\Omega).$ 
By Theorem \ref{thm: solution Ep} it suffices to show that $u\in L^q(\mu)$ which is, in turn, equivalent to showing that
$$
\int_1^{+\infty} \mu(u<-t^{1/q})dt<+\infty. 
$$
The latter follows easily from (\ref{eq: application 1}) and Lemma \ref{lem: capacity}, which completes the proof.
\end{proof}
The exponent $q(p)=\frac{nm(p-1)}{n-mp}$ is sharp in view of the following example.
\begin{examp}\label{example: radial functions}
Consider $\f_{\alpha}=1-\Vert z\Vert^{-2\alpha},$ where $\alpha$ is a constant in $(0,\frac{n-m}{m}).$ 
An easy computation shows that $\f_{\alpha}\in \Fco$ and 
$$
H_m(\f_{\alpha})=C.\Vert z\Vert^{-2m(\alpha+1)}dV=f_{\alpha}dV.
$$ 
Then 
$$
\f_{\alpha}\in \Fc_m^q(\Omega) \Longleftrightarrow q<\frac{n-m}{\alpha}-m,
$$
while 
$$
f_{\alpha}\in L^p(\Omega,dV) \Longleftrightarrow p<\frac{n}{m(\alpha+1)}.
$$
\end{examp}
\medskip

\noindent{\bf Acknowledgements.} {
The paper is taken from my Ph.D Thesis defended on 30th November 2012. It is a great pleasure to express my deep gratitude to my advisor Ahmed Zeriahi for inspirational discussions and enlightening suggestions.  I am also indebted to Vincent Guedj for constant help and encouragements. I also would like to thank Urban Cegrell for very useful discussions on the decomposition theorem. I also thank the reviewer for his/her thorough review and highly appreciate the comments and suggestions, which significantly contributed to improving the quality of the paper. 
}

\bigskip

\noindent \address{Lu Hoang Chinh\\ Chalmers University of Technology\\ Mathematical Sciences\\
SE- 412 96 Gothenburg, \\
Sweden}

\noindent \email{chinh@chalmers.se}
\bigskip





\begin{thebibliography}{99}
\bibitem{AH} D. R. Adams, L.I.Hedberg, "Function spaces and potential theory", Berlin, Springer Verlag, 1996.
\bibitem{ACC10} P. Ahag, U. Cegrell, R. Czyz, {\it On Dirichlet's principle and problem}, Math. Scand. 110 (2012), no. 2, 235-250.
\bibitem{BT76} E. Bedford, B. A. Taylor, {\it The Dirichlet problem for a complex Monge-Amp\`ere equation}, Invent. Math. 37 (1976), no. 1, 1-44.  
\bibitem{Bl05}  Z. Blocki, {\it Weak solutions to the complex Hessian equation}, Ann. Inst. Fourier (Grenoble) 55 (2005), no. 5, 1735-1756.
\bibitem{BBGZ13} R. J. Berman, S. Boucksom, V. Guedj,  A. Zeriahi, {\it A variational approach to complex Monge-Amp\`ere equations}, Publ. Math. Inst. Hautes \' Etudes Sci. 117 (2013), 179-245.
\bibitem{Carleson} L. Carleson, "Selected problems on exceptional sets", Princeton, N.J.-Toronto, Ont. London 1967.
\bibitem{Ceg98} U. Cegrell, {\it Pluricomplex energy}, Acta Math. 180 (1998), no. 2, 187-217.
\bibitem{Ceg04} U. Cegrell, {\it The general definition of the complex Monge-Amp\`ere operator}, Ann. Inst. Fourier (Grenoble) 54 (2004), no. 1, 159-179.
\bibitem{DK11} S. Dinew, S. Kolodziej,  {\it A priori estimates for complex Hessian equations}, preprint arXiv:1112.3063v1.
\bibitem{DK12} S. Dinew, S. Kolodziej,  {\it Liouville and Calabi-Yau type theorems for complex Hessian equations}, preprint 	arXiv:1203.3995v1.
\bibitem{EGZ11} P. Eyssidieux, V. Guedj, A. Zeriahi, {\it Viscosity solutions to degenerate complex Monge-Amp\`ere equations}, Comm. Pure Appl. Math. 64 (2011), no. 8, 1059-1094.
\bibitem{Gar59} L. G{\aa}rding, {\it An inequality for Hyperbolic Polynomials}, Journal of Mathematics and Mechanics, Vol. 8, No. 6 (1959).
\bibitem{GZ05} V. Guedj, A. Zeriahi, {\it Intrinsic capacities on compact K\"{a}hler manifolds}, J. Geom. Anal. 15 (2005), no. 4, 607-639. 
\bibitem{GZ07} V. Guedj, A. Zeriahi, {\it The weighted Monge-Amp\`ere energy of quasiplurisubharmonic functions}, J. Funct. Anal. 250 (2007), no. 2, 442-482.
\bibitem{Hou09} Z. Hou, {\it Complex Hessian equation on K\"{a}hler manifold}, Int. Math. Res. Not. IMRN 2009, no. 16, 3098-3111.
\bibitem{HMW10} Z. Hou, X. Ma, D. Wu, {\it A second order estimate for complex Hessian equations on a compact K\"{a}hler manifold}, Math. Res. Lett. 17 (2010), no. 3, 547-561.
\bibitem{Jb10} A. Jbilou, {\it \' Equations hessiennes complexes sur des vari\'et\'es k\"{a}hl\'eriennes compactes}, C. R. Math. Acad. Sci. Paris 348 (2010), no. 1-2, 41-46.
\bibitem{Kl91} M. Klimek, "Pluripotential theory", London Mathematical Society Monographs. New Series, 6. Oxford Science Publications. The Clarendon Press, Oxford University Press, New York, 1991.
\bibitem{Kok10} V. N. Kokarev, {\it Mixed volume forms and a complex equation of Monge-Amp\`ere type on K\"{a}hler manifolds of positive curvature}, Izv. RAN. Ser. Mat. 74:3 (2010), 65-78.
\bibitem{Kol05} S. Kolodziej, "The complex Monge-Amp\`ere equation and pluripotential theory", Memoirs Amer. Math. Soc. 178 (2005) 64p.
\bibitem{Li04} S. Y. Li, {\it On the Dirichlet problems for symmetric function equations of the eigenvalues of the complex Hessian}, Asian J. Math. 8 (2004),no. 1, 87-106.
\bibitem{Labutin} D. A. Labutin, {\it Pluripolarity of sets with small Hausdorff measure},
Manuscripta Math. 102 (2000), no. 2, 163-167. 
\bibitem{Chinh1} H. C. Lu, {\it Solutions to degenerate complex Hessian equations}, 	Journal de math\' ematiques pures et appliqu\' ees 100 (2013) pp. 785-805 .
\bibitem{Chinh2} H. C. Lu, {\it Viscosity solutions to  complex Hessian equations},  Journal of Functional Analysis 264 (2013) pp. 1355-1379.
\bibitem{Mattila} P. Mattila, "Geometry of sets and measures in Euclidean spaces,  
Fractals and rectifiability". Cambridge Studies in Advanced Mathematics, 44. Cambridge University Press, Cambridge, 1995. xii+343 pp.
\bibitem{Cuong1} N. C. Nguyen, {\it Subsolution theorem for the complex Hessian equation},  to appear on Universitatis Iagellonicae Acta Mathematica. 
\bibitem{Cuong2} N. C. Nguyen, {\it H\"{o}lder continuous solutions to complex Hessian equations}, arXiv:1301.0710.
\bibitem{Per99} L. Persson, {\it A Dirichlet principle for the complex Monge-Amp\`ere operator}, Ark. Mat. 37 (1999), no. 2, 345-356.
\bibitem{Rai69} J. Rainwater, {\it A note on the preceding paper}, Duke Math. J. 36 (1969) 799-800. 
\bibitem{SA12} A. S. Sadullaev, B. I. Abdullaev, {\it Potential theory in the class of $m$-subharmonic functions}, Trudy Matematicheskogo Instituta imeni V.A. Steklova, 2012, Vol. 279, pp. 166-192.
\bibitem{Y} S. T. Yau,  {\it On the Ricci curvature of a compact K\"{a}hler manifold and the complex Monge-Amp\`ere equation}, Comm. Pure Appl. Math. 31 (1978), no. 3, 339-411.
\bibitem{Ziemer} W. P. Ziemer, "Weakly differentiable functions,
Sobolev spaces and functions of bounded variation". Graduate Texts in Mathematics, 120. Springer-Verlag, New York, 1989. xvi+308 pp.
\bibitem{W09} X. J. Wang, {\it The k-Hessian equation}, Lect. Not. Math. 1977 (2009).
\bibitem{YW10} Y. Wang, {\it A viscosity approach to the Dirichlet Problem for Complex Monge-Amp\`ere Equations}, 	Math. Z. 272 (2012), no. 1-2, 497-513.
\end{thebibliography}
\end{document}